\documentclass[12pt]{amsart}

\usepackage{mathrsfs}

\usepackage{mathtools}
\usepackage{amsthm,bm}

\usepackage{comment}

\usepackage{amsfonts,amsmath,amssymb,amsxtra,anysize,times,tikz, float, latexsym}

\usetikzlibrary{arrows,decorations.markings,decorations.pathmorphing}
\usepackage{forest}

\usepackage{tabularx, booktabs}
\newcolumntype{Y}{>{\centering\arraybackslash}X}

\usepackage{subfigure}
\usepackage{xcolor}
\usepackage{multicol}
\usepackage{scalefnt}

\definecolor{mygreen}{rgb}{0.0, 0.5, 0.0}

\author{Eugene Gorsky}
\address{University of California at Davis, Davis, California, US}
\address{International Laboratory of Representation Theory and Mathematical Physics, NRU-HSE, Moscow, Russia}
\email{egorskiy@math.ucdavis.edu}
\author{Mikhail Mazin }
\address{Kansas State University, Manhattan, Kansas, US}
\email{mmazin@math.ksu.edu}
\author{ Monica Vazirani} 
\address{University of California at Davis, Davis, California, US}
\email{vazirani@math.ucdavis.edu}
\title{Recursions for rational $q,t$-Catalan numbers}


\newtheorem{lemma}{Lemma}[section]
\newtheorem{theorem}[lemma]{Theorem}
\newtheorem{corollary}[lemma]{Corollary}
\newtheorem{proposition}[lemma]{Proposition}
\theoremstyle{definition} 
\newtheorem{example}[lemma]{Example}
\newtheorem{remark}[lemma]{Remark}
\newtheorem{definition}[lemma]{Definition}

\newcommand{\dinv}{\mathtt{dinv}}  
\newcommand{\codinv}{\mathtt{codinv}}  
\newcommand{\Cogen}{\mathtt{Cogen}}  
\newcommand{\Ngen}{\mathtt{Ngen}}  
\renewcommand{\a}{\mathtt{g}} 
\newcommand{\NgenD}{\mathtt{Ngen}(\Delta)}  
\newcommand{\MgenD}{\mathtt{Mgen}(\Delta)}  
\newcommand{\Mgen}{\mathtt{Mgen}}

\newcommand{\area}{\mathtt{area}}

\DeclareMathOperator{\Dyck}{Dyck}

\newcommand{\BZ}{\mathbb{Z}}

\newcommand{\ZZ}{\mathbb{Z}_{\ge 0}}

\newcommand{\Z}{\mathbb{Z}}
\newcommand{\MV}[1]{\marginpar{\tiny{{\color{red} {#1} }}}}
\newcommand{\MVlong}[1]{{\tiny{{\color{red} {#1} }}}}

\newcommand{\bu}{\bm{u}}  
\newcommand{\bv}{\bm{v}}
\newcommand{\bw}{\bm{w}}
\newcommand{\bx}{\bm{x}}
\newcommand{\by}{\bm{y}}
\newcommand{\bz}{\bm{z}}
\newcommand{\Pu}{P_{\bm{u}}}

\newcommand{\IMN}{I_{M,N}} 
\newcommand{\IMNo}{I_{M,N}^0} 
\newcommand{\one}{{1^{M+N}}} 
\newcommand{\zero}{{0^{M+N}}} 
\newcommand{\zeroMN}{{0^{M},0^{N}}} 

\newcommand{\ngapsD}{\ngapsD{\Delta}}

\newcommand{\gapsD}{\gapsD{\Delta}}

\usepackage{todonotes}

\begin{document}
\begin{abstract}
We give a simple recursion labeled by binary sequences which computes rational $q,t$-Catalan power series, both in relatively prime and non relatively prime cases. It is inspired by, but not identical to recursions due to B. Elias, M. Hogancamp, and A. Mellit, obtained in their study of link homology. We also compare our recursion with the Hogancamp-Mellit's recursion and verify a connection between the Khovanov-Rozansky homology of $N,M$-torus links and the rational $q,t$-Catalan power series for general positive $N,M.$
\end{abstract}
\date{\today}
\maketitle


\section{Introduction}

In the last decade the rational $q,t$-Catalan numbers attracted a lot of interest in algebraic combinatorics. Given a pair of integers $(M,N)$, we can consider the set of all partitions which are simultaneously $M$- and $N$-cores, that is, none of their hook lengths are divisible by $M$ or $N$. It is easy to see (e.g. \cite{GM2}) that such $(M,N)$-cores are in bijection with the subsets $\Delta\subset\BZ_{\ge 0}$ such that
$0 \in \Delta$, $\Delta+N\subset \Delta,\Delta+M\subset\Delta$ and $\overline{\Delta}:=\BZ_{\ge 0}\setminus\Delta$ is finite. We will relax the normalization condition $0\in\Delta$ and call such subsets \textit{$(M,N)$--invariant}.

If $M$ and $N$ are coprime, then  Anderson \cite{Anderson} proved that the set of $(M,N)$ cores is finite and, in fact, is in bijection with the set $\Dyck(M,N)$ of Dyck paths in the $M\times N$ rectangle. For such paths one can define two statistics $\area$ and $\dinv$ and define a bivariate polynomial
$$
c_{M,N}(q,t)=\sum_{D\in \Dyck(M,N)}q^{\area(D)}t^{\dinv(D)}.
$$
This polynomial generalizes $q,t$-Catalan numbers of Garsia and Haiman \cite{GaHa}  (which appear at $M=N+1$)
 and has 
lots of remarkable properties, for example, it is symmetric in $q$ and $t$. The latter follows from the so-called rational Shuffle conjecture \cite{GN,BGLX} recently proved by Mellit \cite{Mellit1}.  The statistic $\dinv$ has several equivalent definitions (see Definition \ref{def:codinv} below); the most elegant one is obtained using the sweep map of Armstrong et.~al.~\cite{ALW,AHJ}. Using the above bijections, one can   translate $\dinv$ as a statistics on $(M,N)$-invariant subsets, which was explicitly defined in \cite{GM1}, see Section \ref{sec:recursion} for details. Thus, 
\begin{equation}
\label{def cMN}
c_{M,N}(q,t)=\sum_{\Delta\in \IMNo}q^{\area(\Delta)}t^{\dinv(\Delta)}=(1-q)\sum_{\Delta\in \IMN}q^{\area(\Delta)}t^{\dinv(\Delta)}.
\end{equation}
where $\IMN$ (respectively, $\IMNo$) denotes the set of $(M,N)$ invariant subsets (with
$0 \in \Delta$). 

If $M$ and $N$ are not coprime, then the sets of $(M,N)$ cores and invariant subsets  are still in bijection and are infinite,
but the relation between them and Dyck paths is more involved. Still, in \cite{GMV16} the authors defined a surjection from $\IMNo$ to $\Dyck(M,N),$ such that the $\dinv$ statistic is constant on the fibers, and the $\area$ statistic behaves in a natural and easily controlled way. In this case one can define $c_{M,N}(q,t)$ by the same equation \eqref{def cMN}. However,  $c_{M,N}(q,t)$ is no longer a polynomial
but a power series. In fact, we will show that it is a rational function with denominator $(1-q)^{d-1}$, where $d =\gcd(M,N)$. 

The work of A. Mellit on Shuffle conjecture can be extended to show that the polynomial $(1-q)^{d-1}c_{M,N}(q,t)$ is symmetric in $q$ and $t$ in the non relatively prime case as well. However, to our knowledge, this did not appear in the literature yet. Note that in the non relatively prime case the coefficients of $(1-q)^{d-1}c_{M,N}(q,t)$ are not necessarily positive anymore (see Examples \ref{ex: 22}, \ref{ex: 33}, and \ref{ex-46}).

One of the most remarkable properties of $c_{M,N}(q,t)$ is its connection to {\em Khovanov-Rozansky homology} of  $(M,N)$ torus links conjectured in \cite{Gor,GN,GNR} and proved by Elias, Hogancamp and Mellit in a series of papers \cite{EH,Hog,Mellit2,HogMellit} in various special cases. See section \ref{sec:homology} for a precise statement.
In short, the comparison between the power series  $c_{M,N}(q,t)$ and the Poincar\'e power series of this homology is proved by obtaining certain recursions on the topological side and then verifying them on combinatorial side. The terms in these recursions are labeled by binary sequences of varying length.

The main objective of this paper is to understand these recursions as clearly as possible in combinatorial terms. 
Given an $(M,N)$ invariant subset $\Delta$, we consider a length $M+N$ binary sequence $\bu=\bu(\Delta)$ recording the characteristic function of the intersection $\Delta\cap [0,M+N-1]$. We define
$$
P_{\bu}(q,t)=\sum_{\Delta\in I_{M,N}, \bu(\Delta)=\bu}q^{\area(\Delta)}t^{\codinv(\Delta)}.
$$
where 
\begin{equation*}
\codinv(\Delta)=\delta(N,M)-\dinv(\Delta),
\end{equation*}
and 
\begin{equation}
\label{eq:delta}
\delta(N,M):=\frac{NM-N-M+\gcd(M,N)}{2}
\end{equation} 
is the maximal possible value of $\dinv.$
In Theorem \ref{Theorem: recursion binary N+M} we prove a simple recursion for the power series $P_{\bu}(q,t)$. 
In Theorem \ref{Theorem: recursion unique} we prove that this recursion has a unique solution given the initial condition
$P_{\one}(q,t)=1$. These results hold both for coprime and non-coprime $(M,N)$. We also observe in Lemma \ref{lem: zero and catalan} that
\begin{equation}
\label{eq: P to catalan}
t^{\delta(N,M)}P_{\zero}(q,t^{-1})=\frac{q^{M+N}}{1-q}c_{M,N}(q,t),
\end{equation}
and hence the function $c_{M,N}(q,t)$ can be computed using this recursion. 

In Section \ref{sec:examples} we write complete decision trees for this recursion for $(M,N)=(2,2),(3,3)$ and $(4,6)$,
and compute the corresponding rational Catalan series. 

In Section \ref{sec:comparison} we compare our recursion with the ones appearing in \cite{EH,Hog,Mellit2,HogMellit}.
One important distinction is that our recursion is labeled by binary sequences of fixed length $M+N$ while their recursion 
is labeled by pairs of binary sequences of varying length. Still, we prove that the recursions are very similar, and the resulting expressions for $c_{M,N}(q,t)$ agree. In Section \ref{sec-new-higher-a} we  add higher $a$-degrees and give recursions for the rational $q,t$--Schr\"oder power series.

\section*{Acknowledgments}

The authors would like to thank Ben Elias, Matt Hogancamp and Anton Mellit for patiently explaining their work to us and 
sharing the early drafts of \cite{HogMellit}. The work of E. G. was partially supported by the NSF grants DMS-1700814 , DMS-1760329 and the Russian Academic Excellence Project 5-100.
The work of M. M. was partially supported by the Simons Foundation Collaboration Grant for Mathematicians, award number 524324.
The work of M.~V.~was partially supported by the Simons Foundation Collaboration Grant for Mathematicians, award number 319233.
The authors also thank the
NSF Focused Research Group ``Algebra and Geometry Behind Link Homology" 
 for financial support and hosting their participation
in the ``Hilbert schemes, categorification and combinatorics'' workshop
at Davis.


\section{The recursion}
\label{sec:recursion}

Let $(M,N)=(dm,dn)$ be a pair of positive integers, where $m$ and $n$ are relatively prime, so $d=\gcd(M,N).$ 

\begin{definition}
The set $\IMN$ of $M,N$-invariant subsets is defined by
\begin{equation*}
I_{M,N}:=\{\Delta\subset\BZ_{\ge 0}: \Delta+N\subset \Delta,\Delta+M\subset\Delta,\sharp\overline{\Delta}<\infty\},
\end{equation*}
where $\Delta+N$ denotes the shift of $\Delta$ by $N,$ i.e. 
\begin{equation*}
\Delta+N:=\{k\in \BZ: k-N\in\Delta\},
\end{equation*}
$\overline{\Delta}:=\BZ_{\ge 0}\setminus\Delta$ is the complement to $\Delta,$ and $\sharp\overline{\Delta}$ is the number of elements in $\overline{\Delta}.$ The elements of $\overline{\Delta}$ are often called \textit{gaps} in $\Delta$.
\end{definition}


We define statistics $\area$ and $\codinv$ on the invariant subsets. The $\area$ statistic simply counts the number of gaps in $\Delta:$

\begin{definition}
We set 
\begin{equation*}
\area(\Delta):=\sharp\overline{\Delta}.
\end{equation*}
\end{definition}
The statistics $\dinv$ and $\codinv$ are  more involved.

\begin{definition}
Let $\Delta\in I_{M,N}$ be an invariant subset. The set $\Ngen(\Delta)$ of $N$-generators of $\Delta$ is defined by 
\begin{equation*}
\Ngen(\Delta):=\Delta\setminus(\Delta+N)=\{\a\in\Delta:\a-N\notin\Delta\}.
\end{equation*}
The $M$-generators are defined similarly:
\begin{equation*}
\Mgen(\Delta):=\Delta\setminus(\Delta+M)=\{\a\in\Delta:\a-M\notin\Delta\}.
\end{equation*}
\end{definition} 


\begin{remark}
The condition $\sharp\overline{\Delta}<\infty$ implies that $\sharp\Ngen(\Delta)=N$, one $N$-generator in each congruence class modulo $N$. 
\end{remark}

\begin{definition}
\label{def:codinv}
We set
\begin{equation*}
\codinv(\Delta)=\sum\limits_{\a\in\Ngen(\Delta)} [\a,\a+M-1]\cap\overline{\Delta},
\qquad \dinv = \delta(N,M) - \codinv(\Delta),
\end{equation*}
where  $\delta(N,M)$ is as in \eqref{eq:delta} and 
we use the \textit{integer interval} notation:
\begin{equation*}
[\a,\a+k]:=\{\a,\a+1,\ldots,\a+k\}.
\end{equation*}
\end{definition}

\begin{remark}
One can check (see e.g \cite{GM1,GM2}) that the definition of $\codinv(\Delta)$
is in fact symmetric in $M$ and $N$.
\end{remark}


\begin{definition}
Let ${\bu}=(u_0,\ldots,u_{N+M-1})$ be a sequence of $0$'s and $1$'s. We set
\begin{equation*}
I_{\bu}:=\{\Delta\in I_{M,N}: \forall \, 0\le i< N+M,\ i\in\Delta \Leftrightarrow u_i=1\}.
\end{equation*}
We say that a sequence $\bu$ is \textit{admissible} if $I_{\bu}\neq \emptyset.$ Note that we number the entries of $\bu$ starting at $0$.
\end{definition}

\begin{definition}
Let the power series $\Pu(q,t)$ be given by
\begin{equation*}
P_{\bu}:=\sum\limits_{\Delta\in I_{\bu}} q^{\area(\Delta)}t^{\codinv(\Delta)}.
\end{equation*}
\end{definition}

\begin{remark}
Note that while the set $I_{\bu}$ is often infinite, the sets $\{\Delta\in I_{\bu}: \area(\Delta)=k\}$ are always finite. Therefore, $P_{\bu}$ is a well defined power series in $q$ and $t$ with positive integer coefficients. 
Observe $\Pu=0$ if $\bu$ is not admissible.
\end{remark}

\begin{lemma}
\label{lem: zero and catalan}
One has 
$$
P_{\zero}(q,t)=q^{M+N}\sum_{\Delta\in \IMN}q^{\area(\Delta)}t^{\codinv(\Delta)},
$$
and
$$
c_{M,N}(q,t)=q^{-N-M}t^{\delta(N,M)}(1-q)P_{\zero}(q,t^{-1}).
$$
\end{lemma}
\begin{proof}
Indeed, for $\bu=\zero$ the set $I_{\bu}$ consists of all $(M,N)$--invariant subsets which do not intersect with 
$[0,M+N-1]$. All such subsets are obtained from $(M,N)$--invariant subsets in $\IMN$ by shift by $(M+N)$. 
It is easy to see that the shift does not change $\codinv$ and changes $\area$ by $(M+N)$. 
The second formula now follows from Equation \eqref{def cMN} and the relation $\dinv(\Delta)+\codinv(\Delta)=\delta(N,M)$.
\end{proof}

\begin{definition} \label{def:rho}
Define
$\rho:\IMN\to \IMN$
to be the shift map
given by
$$
\rho(\Delta)=
\begin{cases}
\Delta-1 & \text{if } 0\notin\Delta
\\
(\Delta\setminus\{0\})-1 & \text{if } 0\in\Delta.
\end{cases}
$$
\end{definition}
\begin{definition} \label{def:lambda}
Let $\bu \in \{0,1\}^{M+N}$. We define
$$\lambda({\bu}):=\sum\limits_{i=0}^{M-1}(u_{i+N}-u_i).$$
If $\Delta \in I_{\bu}$ we set $\lambda(\Delta):= \lambda(\bu)$.
\end{definition}
\begin{remark}
\label{rem: lambda}
Note that if $\Delta \in I_{\bu}$ , then
$\lambda({\bu})$ counts the $N$-generators of $\Delta$
 in the interval $[N,N+M-1],$ or, equivalently, the number of $M$-generators in the interval $[M,N+M-1].$ Indeed,
\begin{equation*}
\begin{aligned}
\lambda(\bu)&=\sum\limits_{i=0}^{M-1}(u_{i+N}-u_i)=\sum\limits_{j=N}^{N+M-1} u_j -\sum\limits_{i=0}^{M-1} u_i\\
&=\sum\limits_{j=M}^{N+M-1} u_j -\sum\limits_{i=0}^{N-1} u_i=\sum\limits_{i=0}^{N-1}(u_{i+M}-u_i),
\end{aligned}
\end{equation*}
as the $u_i$ for $i\in [\min(N,M),\max(N,M)-1]$ cancel out. Also, clearly,
\begin{equation*}
\sum\limits_{i=0}^{M-1}(u_{i+N}-u_i)=\sharp(\Ngen(\Delta)\cap [N,N+M-1]),
\end{equation*}
and
\begin{equation*}
\sum\limits_{i=0}^{N-1}(u_{i+M}-u_i)=\sharp(\Mgen(\Delta)\cap [M,N+M-1]).
\end{equation*}
\end{remark}

\begin{theorem}\label{Theorem: recursion binary N+M}
Let $\bu=(u_0,\ldots,u_{N+M-1})$ be an admissible sequence. Let also 
\begin{equation*}
\begin{aligned}
&\bv=(u_1,\ldots,u_{N+M-1},1),\\
&\bv'=(u_1,\ldots,u_{N+M-1},0).
\end{aligned}
\end{equation*}
The power series $\Pu$ satisfy the following recurrence relation:
\begin{equation}
\label{eq-recursion binary N+M}
\Pu=
\begin{cases} 
q(P_{\bv}+P_{{\bv'}}),& \text{if} \ \ u_0=u_N=u_M=0,\\
qP_{\bv},& \text{if} \ \ u_0=0 \ \  \text{and} \ u_N+u_M>0,\\
t^{\lambda(\bu)}P_{\bv},& \text{if} \ \ u_0=u_N=u_M=1,
\end{cases}
\end{equation}
where $\lambda({\bu}):=\sum\limits_{i=0}^{M-1}(u_{i+N}-u_i).$
\end{theorem}

\begin{proof}
With respect to the statistics above, the shift map $\rho$
of Definition \ref{def:rho} has the following properties:
\begin{enumerate}
\item[(a)] If $0\notin \Delta,$ then $\area(\rho(\Delta))=\area(\Delta)-1,$ while if $0\in \Delta,$ then $\area(\rho(\Delta))=\area(\Delta).$
\item[(b)] If at least one of the numbers $N$ and $M$ belongs to $\Delta,$ then $N+M-1\in\rho(\Delta),$ while if neither $N$ nor $M$ are in $\Delta,$ then either possibility $N+M-1\in\rho(\Delta)$ or $N+M-1\notin\rho(\Delta)$ may occur.
\item[(c)] If $0\notin \Delta,$ then $\codinv(\rho(\Delta))=\codinv(\Delta).$
\item[(d)] If $0\in \Delta,$ then 
\begin{equation*}
\codinv(\rho(\Delta))=\codinv(\Delta)-\sharp\left([0,M-1]\cap\overline{\Delta}\right)+\sharp\left([N,N+M-1]\cap\overline{\Delta}\right).
\end{equation*}
\end{enumerate}
To prove part (d), one should observe that all the $N$-generators of $\Delta,$ except $0,$ are simply shifted down by one in $\rho(\Delta),$ while retaining the same contributions to $\codinv.$ The $N$-generator $0\in\Delta$ get replaced by $N-1\in\rho(\Delta).$ The contribution to $\codinv$ changes accordingly,
and this change is measured by $\lambda(\bu)$ for $\Delta \in I_{\bu}$. 
\end{proof}

\begin{definition}
\label{def: decision tree}
We visualize the recursion \eqref{eq-recursion binary N+M} using the {\it decision tree}.
Each node corresponds to a binary sequence $\bu$ and the edges connect $\bu$ with $\bv$ and $\bv'$
and are labeled by the corresponding coefficients:
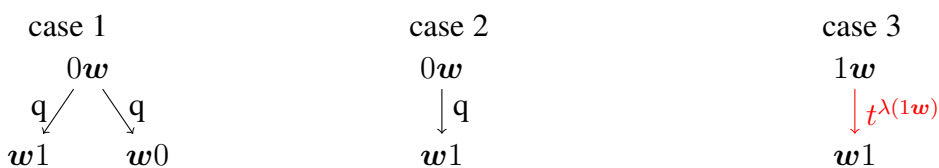
\begin{center}
\begin{tabular}{p{2in} p{2in} p{2in}}
\quad case 1 & case 2 & case 3 \\
\begin{forest}
[$0\bw$, s sep=7mm
[$\bw 1$,  edge label={node[left,midway]{q}}, edge=->]
[$\bw 0$,  edge label={node[right, midway]{q}},edge=->]]
\end{forest}
&
\begin{forest}
[$0 \bw$
[$\bw 1$,  edge label={node[right,midway]{q}},edge=->]
]
\end{forest}
&
\begin{forest}
[$1\bw$
[$\bw 1$,  edge=red, ,edge=->,
	edge label={node[right,midway,red]{$t^{\lambda(1\bw)}$}}]]
\end{forest}
\end{tabular}
\end{center}
Here $\bw \in \{0,1\}^{M+N-1}$, $\bu=0\bw$ in cases 1 and 2 and $\bu=1\bw$ in case 3,
$\bv=\bw1$ and $\bv'=\bw0$. 
Note that we can view case 2 as a special instance of case 1 for which
$\bw 0$ is {\it not} admissible and so $P_{\bw 0}=0$.
We color edges and labels in case 3 in {\textcolor{red}{red}}
to emphasize that these carry powers of $t$ while
all other (black) edges are labeled by $q$.
\end{definition}

\begin{remark}
\label{rem: decision tree}
If we never identify vertices with the same label, we will
indeed get an infinite tree. However, it is convenient to make the graph finite by 
keeping each $\one$ as a terminal vertex, and identifying the pairs of vertices with the same label, whenever one vertex is a predecessor of another. This leads to directed cycles, which we analyze below. See also 
the examples in Section \ref{sec:examples}.
\end{remark}

\begin{definition}
\label{def-periodic}
We will call $\bu \in \{0,1\}^{M+N}$ {\em $p$-periodic} if
for all $i\in \Z$, $\tilde u_i = \tilde u_{i+p}$, where
$\tilde \bu$ is the infinite sequence formed via
$\tilde u_{i +r(M+N)} := u_i$ for $r \in \Z$, $0 \le i < M+N$.
\end{definition}
\begin{lemma}
\label{lemma-periodic}
For $\bu$ admissible, $\lambda(\bu)=0$ if and only if $\bu$ is $p$-periodic
for some $p \mid \gcd(M,N)$.
\end{lemma}
\begin{proof}
The $\Leftarrow$ direction is clear. For the $\Rightarrow$ direction, it suffices
to show
$\bu$ is both $M$-periodic and $N$-periodic.

First suppose $0\le i < M$.
Recall $\bu$ being admissible means $u_{i+N}= 0 \implies u_i  = 0$.  
Suppose $0=u_i = \tilde u_i$.
Then $0 =\lambda(\bu) = \sum_{i=0}^{M-1} (u_{i+N}-u_i)$ forces
$\tilde u_{i+N} = u_{i+N}=0$ as well. 

Next suppose $M\le i < M+N$, so $0 \le i-M < N$.
If $0 = \tilde u_i = u_i = u_{(i-M)+M}$,
then the admissibility of $\bu$
 implies $0 =  u_{i-M}= \tilde u_{i+N}$.
On the other hand, if $0 = \tilde u_{i+N}=  u_{i-M}$
then $0 =\lambda(\bu) = \sum_{i=0}^{N-1} (u_{i+M}-u_i)
= \sum_{i=M}^{M+N-1} (u_{i}-u_{i-M}) $ forces
$\tilde u_{i} = u_{i}=0$ as well. 
These arguments show that for all $i$, $\tilde u_i = 0 \iff \tilde u_{i+N}=0$ and so $\bu$ is $N$-periodic. 
A similar argument shows $\bu$ is $M$-periodic.
\end{proof}

\begin{theorem}
\label{Theorem: recursion unique}
The recursion in Theorem \ref{Theorem: recursion binary N+M} has a unique solution given the initial condition 
$P_{\one}(q,t)=1.$ Moreover, for any sequence $\bu$ the power series $P_{\bu}(q,t)$ can be expressed as a rational function with the denominator $\prod_{i=1}^d (1-q^{\ell_i}),$ where $0<\ell_i<d$ for all $i.$ 
\end{theorem}

\begin{proof}
We prove the statement by induction on the number $k$
 of zeroes in the sequence $\bu$. If there are no zeroes, i.~e.~ $k=0$,
 we have
$\bu = \one$ and 
$P_{\one}=1$. Assume now that a sequence $\bu$ has $k$ zeroes and
$k \ge 1$. Assume also that $\bu$ is admissible, since otherwise
$\Pu = 0$. Note that no cases of the recursion increase the number of $0$s in the sequence, and the case 2 recursion steps always decrease it.

Let us apply the recursion relation to $P_{\bu}.$ If it is a case 2 relation, then the number of $0$'s decreased, and we can determine $P_{\bu}$ by induction. Otherwise, we get exactly one term with exactly $k$ zeros in the next step of the recursion. Note that the corresponding sequence is simply a cyclic rotation of $\bu:$\ $0\bw\to \bw 0$ in case 1 and $1\bw\to\bw 1$ in case 3. Let us keep applying the recursion to the terms with $k$ zeros, until either there are no such terms, and $P_{\bu}$ can be determined by induction, or the term $P_{\bu}$ is repeated. It is not hard to see that in the latter case the sequence $\bu$ is both $M$- and $N$- periodic.
 Indeed, otherwise one would have to use a case 2 recursion relation at some point.

We get the linear equation
\begin{equation}
\label{eq:cycle}
P_{\bu}=\gamma P_{\bu}+\sum_{j \in J} \gamma_j P_{\bw^{(j)}},
\end{equation}
 where the binary sequences $P_{\bw^{(j)}}$ contain  $k-1$ zeros each, $\gamma, \gamma_j$ are some monomials in $q$ and $t,$ and $\gamma\neq 1$.Indeed, since $\bu$ contains $k>0$ zeros, a case 1 relation must occur at least once. By the inductive hypothesis, we can compute the series $P_{\bw^{(j)}}$ for all $j\in J$, and then we can solve \eqref{eq:cycle} to obtain $P_{\bu}$.

Finally, one can show that the constant $\gamma$ in \eqref{eq:cycle} is just a power of $q$. Indeed, let $p|\gcd(M,N)$ be the period of $\bu$.
Then by Lemma \ref{lemma-periodic}, $\lambda(\bu)=0,$
hence all the case 3 edges carry weight $1 = t^0$.
Finally, if $\bu$ contains $k$ zeroes then the number of zeroes in the period equals $\frac{kp}{M+N}$,
so
$
\gamma=q^{\frac{kp}{M+N}} 
$
and from \eqref{eq:cycle} we get
$$
P_{\bu}=\frac{1}{1-q^{\frac{kp}{M+N}}}\sum_{j\in J} \gamma_j P_{\bw^{(j)}}.
$$

Note that $0<\frac{kp}{M+N}<d,$ and that one will use the equation \ref{eq:cycle} in the computation of $P_{\bu}(q,t)$ at most $d$ times. Indeed, this equation is only applied when $\bu$ is periodic with period dividing $d,$ and each time it is applied the number of $1$'s in the sequence increases. Therefore, in the end (before reducing) one gets a rational function with the denominator equal to
$\prod_{i=1}^d (1-q^{\ell_i}),$ where $0<\ell_i<d$ for all $i.$

\end{proof}

We will see in Section \ref{sec:comparison}, that the denominator of $P_{0^{N+M}}$ can, in fact, be reduced to $(1-q)^d.$

\section{Examples}
\label{sec:examples}

In this section we present some examples of decision trees defined in Definition \ref{def: decision tree} and Remark \ref{rem: decision tree}. 
As all black edges have weight $q$, we can drop this label.
Further, it is sometimes convenient to just record the (new) rightmost
entry at each node, simplifying the picture in Definition \ref{def: decision tree} as follows: 

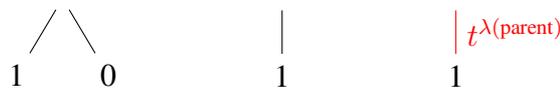
\begin{center}
\begin{forest}
[{}, s sep=7mm [1] [0]]
\end{forest}
\qquad
\qquad
\begin{forest}
[{} [1] ]
\end{forest}
\qquad
\qquad
\begin{forest}
[{} [1,  edge=red, edge label={node[right,midway,red]
	{$t^{\lambda(\text{parent} )}$}}]]
\end{forest}
\end{center}
Also, we will replace all branches with a single terminal vertex by the corresponding monomial.
We will refer to the result as to ``compact decision tree''.

\begin{example}
\label{ex: 22}
The decision tree  for $(M,N)=(2,2)$ is shown in Figure \ref{fig: 22}.
We immediately compute
$$
P_{1011}=qt,\quad P_{0101}=q(P_{1011}+P_{0101})=q^2t+qP_{0101},
$$
so
$$
P_{0101}=\frac{q^2t}{1-q}.
$$
Now 
$$
P_{0001}=q^3+q^2P_{0101}=q^3+\frac{q^4t}{1-q}.
$$
Finally, 
$$
P_{0000}=qP_{0001}+qP_{0000},
$$
hence
$$
P_{0000}=\frac{qP_{0001}}{1-q}=\frac{q^4}{1-q}+\frac{q^5t}{(1-q)^2}.
$$
Observe the  $q, t$-symmetry of
$$
(1-q)c_{2,2}(q,t)=q^{-4} t (1-q)^2 P_{0000}(q, t^{-1}) = {q+t-qt}.
$$
\end{example}


\begin{figure}[ht!]
\begin{multicols}{2}
{\scalefont{0.7}
\begin{forest}
for tree={s sep= 2em, l=2ex}
[0000, name=root
[000\color{blue}{1}, edge label={node[right,midway] {$q$}}, edge=->
  [001\color{blue}{1}, edge label={node[left, midway] {$q$}}, edge=-> 
[011\color{blue}{1}, edge label={node[right, midway] {$q$}}, edge=->
[111\color{blue}{1}, edge label={node[right, midway] {$q$}}, edge=->]]]
[001\color{blue}{0}, edge label={node[right, midway] {$q$}}, edge=->
[010\color{blue}{1}, edge label={node[right, midway] {$q$}}, edge=->,
name=loop
   [101\color{blue}{1}, edge label={node[left, midway] {$q$}}, edge=->
 [011\color{blue}{1}, edge=red, edge=->, edge label={node[midway,right, red]{t}} 
[111\color{blue}{1}, edge label={node[right, midway] {$q$}}, edge=->]]]
[101\color{blue}{0}, edge label={node[right, midway] {$q$}}, edge=->, name=end] ]
] ] ]
{\scalefont{1.3}
\draw[->, red] (end) .. controls +(south east:1cm) and +(east:2cm) .. 
node[near start,right, red]{$\;1$} (loop);
\draw[->] (root) .. controls +(south east:2cm) and +(east:2cm) .. 
node[near start,right]{\quad $q$} (root);
}
\end{forest}
}

\columnbreak

{\scalefont{0.7}
\begin{forest}
for tree={s sep= 2em, l=1ex}
[0000, name=root
[\color{blue}{1}
 [\color{blue}{1}\\{$q^2$}, align=right, base=top, name=0011  ]
 [\color{blue}{0} [\color{blue}{1}, name=loop
        [\color{blue}{1}\\{$qt$}, align=right, base=top, name=1011 ] 
                                [\color{blue}{0}, name=end] ]
                        ]
 ] ]
{\scalefont{1.3}
\draw[->, red] (end) .. controls +(south east:1cm) and +(east:2cm) .. 
node[near start,right, red]{$1$} (loop);
\draw[->] (root) .. controls +(south east:2cm) and +(east:2cm) .. 
node[near start,right]{\quad $q$} (root);
}
\end{forest}
}
\end{multicols}
\caption{
Decision tree for $(M,N)=(2,2)$ and the corresponding compact decision tree on the right.}
\label{fig: 22}
\end{figure}
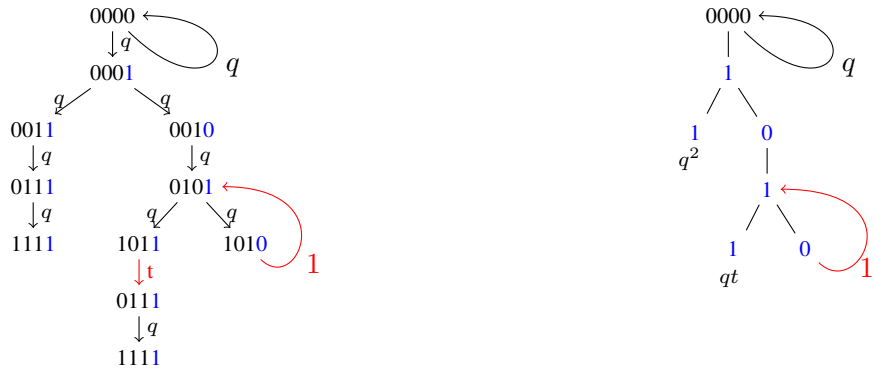


\begin{example}
\label{ex: 33}
The decision tree for the recursion for $(M,N)=(3,3)$ is shown in Figures \ref{fig: 33} and \ref{fig: 33 sub}.
The compact version is shown in Figure \ref{fig: 33 total}.
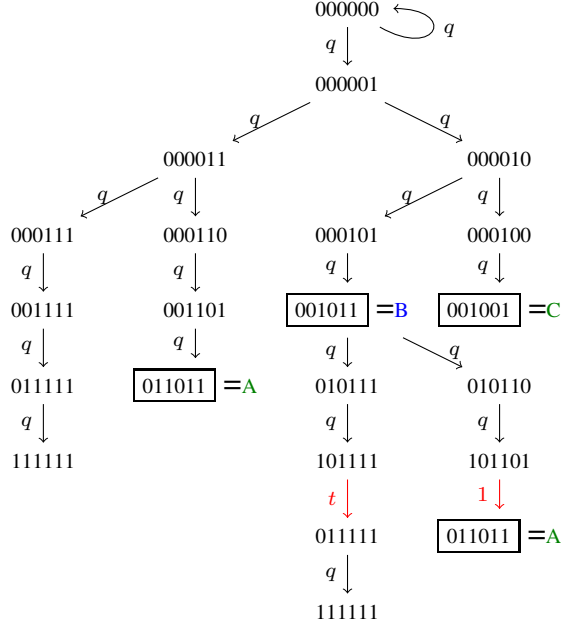
\begin{figure}[ht!]
\begin{tikzpicture}
\node (A) at (10,5)  {\tiny{000000}};
\node (B) at (10,4)  {\tiny{000001}};
\node (C) at (8,3)    {\tiny{000011}};
\node (D) at (12,3)  {\tiny{000010}};
\node (E) at (6,2)  {\tiny{000111}};
\node (F) at (8,2)  {\tiny{000110}};
\node (G) at (10,2)  {\tiny{000101}};
\node (H) at (12,2)  {\tiny{000100}};
\node (I) at (6,1)  {\tiny{001111}};
\node (J) at (8,1)  {\tiny{001101}};
\node (K) at (10,1)  {\fbox{\tiny{001011}} =\color{blue}{\tiny{B}}};
\node (L) at (12,1)  {\fbox{\tiny{001001}} =\color{mygreen}{\tiny{C}}};
\node (M) at (6,0)  {\tiny{011111}};
\node (N) at (8,0)  {\fbox{\tiny{011011}} =\color{mygreen}{\tiny{A}}};
\node (O) at (10,0)  {\tiny{010111}};
\node (P) at (12,0)  {\tiny{010110}};
\node (Q) at (6,-1)  {\tiny{111111}};
\node (R) at (10,-1)  {\tiny{101111}};
\node (S) at (12,-1)  {\tiny{101101}};
\node (T) at (10,-2)  {\tiny{011111}};
\node (U) at (12,-2)  {\fbox{\tiny{011011}} =\color{mygreen}{\tiny{A}}};
\node (V) at (10,-3)  {\tiny{111111}};

\draw 
 (A) edge[in=0, out=330,loop] node[right] {\tiny{$q$}} (A)
    edge[->] node[left] {\tiny{$q$}} (B)  
 (B) edge[->] node[left] {\tiny{$q$}} (C)
    edge[->] node[right]{ \tiny{ $q$}}  (D)
 (C) edge[->] node[left] {\tiny{$q$}}  (E)
      edge[->] node[left] {\tiny{$q$}}  (F)
 (D) edge[->] node[left] {\tiny{$q$}} (G)
       edge[->] node[left]{\tiny{$q$}}  (H)
 (E) edge[->] node[left] {\tiny{$q$}}  (I)
 (F) edge[->] node[left] {\tiny{$q$}}  (J)
 (G) edge[->] node[left] {\tiny{$q$}}  (K)
 (H) edge[->] node[left] {\tiny{$q$}}  (L)
  (I) edge[->] node[left] {\tiny{$q$}}  (M)
 (J) edge[->] node[left] {\tiny{$q$}}  (N)
 (K) edge[->] node[left] {\tiny{$q$}}  (O)
      edge[->] node[right] { \tiny{ $q$}}  (P) 
 (M) edge[->] node[left] {\tiny{$q$}}  (Q)
 (O) edge[->] node[left] {\tiny{$q$}}  (R)
 (P) edge[->] node[left] {\tiny{$q$}}  (S)
 (R) edge[->, red] node[left] {\tiny{$\color{red}{t}$}}  (T)
 (S) edge[->, red] node[left] {\tiny{$\color{red}{1}$}}  (U)
 (T) edge[->] node[left] {\tiny{$q$}}  (V)
;
\end{tikzpicture}
\caption{Decision tree for $(M,N)=(3,3).$ See Figure \ref{fig: 33 sub} for the inserts A and C.}
\label{fig: 33}
\end{figure}
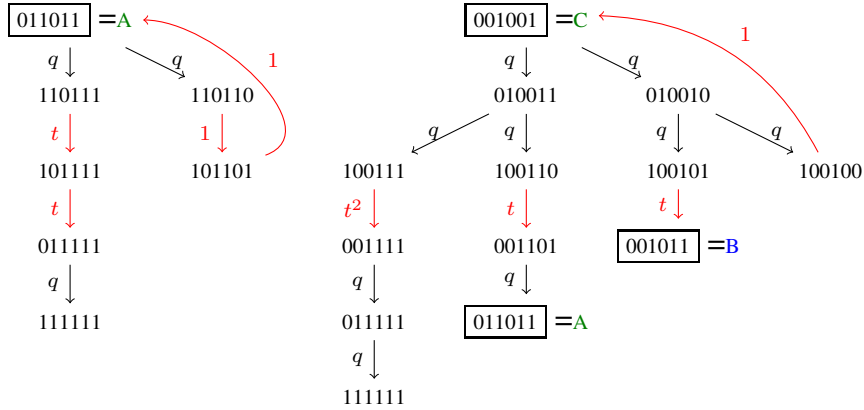
\begin{figure}[ht!]
\begin{tikzpicture}
\node (A) at (2,5)  {\fbox{\tiny{011011}} =\color{mygreen}{\tiny{A}}};
\node (B) at (2,4)  {\tiny{110111}};
\node (C) at (2,3)    {\tiny{101111}};
\node (D) at (2,2)  {\tiny{011111}};
\node (E) at (2,1)  {\tiny{111111}};
\node (F) at (4,4)  {\tiny{110110}};
\node (G) at (4,3)  {\tiny{101101}};
\node (I) at (8,5)  {\fbox{\tiny{001001}} =\color{mygreen}{\tiny{C}}};
\node (J) at (8,4)  {\tiny{010011}};
\node (K) at (10,4)  {\tiny{010010}};
\node (L) at (6,3)  {\tiny{100111}};
\node (M) at (6,2)  {\tiny{001111}};
\node (N) at (6,1)  {\tiny{011111}};
\node (O) at (6,0)  {\tiny{111111}};
\node (P) at (8,3)  {\tiny{100110}};
\node (R) at (8,2)  {\tiny{001101}};
\node (S) at (8,1)  {\fbox{\tiny{011011}} =\color{mygreen}{\tiny{A}}};
\node (T) at (10,3)  {\tiny{100101}};
\node (U) at (10,2)  {\fbox{\tiny{001011}} =\color{blue}{\tiny{B}}};
\node (V) at (12,3)  {\tiny{100100}};

\draw 
 (A) edge[->] node[left] {\tiny{$q$}} (B)
    edge[->] node[right] { \tiny{ $q$}} (F)
 (B) edge[->,red] node[left] {\tiny{$\color{red}{t}$}} (C)
 (C)   edge[->, red] node[left]{\tiny{$\color{red}{t}$}}  (D)
 (D) edge[->] node[left] {\tiny{$q$}}  (E)
  (F)    edge[->, red] node[left] {\tiny{$\color{red}{1}$}}  (G)
 (I) edge[->] node[left] {\tiny{$q$}} (J)
       edge[->] node[right]{ \tiny{ $q$}}  (K)
 (J) edge[->] node[left] {\tiny{$q $} }  (L)
      edge[->] node[left] {\tiny{$q$}}  (P)
 (L) edge[->, red] node[left] {\tiny{$\color{red}{t^2}$}}  (M)
 (M) edge[->] node[left] {\tiny{$q$}}  (N)
  (N) edge[->] node[left] {\tiny{$q$}}  (O)
 (P) edge[->, red] node[left] {\tiny{$\color{red}{t}$}}  (R)
 (R) edge[->] node[left] {\tiny{$q$}}  (S)
  (K)    edge[->] node[left] {\tiny{$q$}}  (T)
      edge[->] node[right] { \tiny{ $q$}}  (V)
 (T) edge[->,red] node[left] {\tiny{$\color{red}{t}$}}  (U)
 (G) edge[in=0,out=20, above, ->, red] node[above right] {\tiny{$\color{red}{1}$}}  (A)
 (V) edge[bend right,->, red] node[above right] {\tiny{$\color{red}{1}$}}  (I)
;
\end{tikzpicture}
\caption{The subgraphs A and C for $(M,N)=(3,3)$ (see Figure \ref{fig: 33} for the main graph).}
\label{fig: 33 sub}
\end{figure}

We first compute the value of the loop
$$
A=P_{011011}=q^2t^2+Aq,\quad A=\frac{q^2t^2}{1-q}.
$$
Next we compute the values of  
$$
B=P_{001011}=q^3t+q^2A=q^3t+\frac{q^4t^2}{1-q},
$$
and
$$
C=q^4t^2+q^3tA+q^2tB+q^2C=q^4t^2+q^3tA+q^5t^2+q^4tA+q^2C,
$$
hence
\begin{equation}
\label{eq: C}
C=\frac{q^4t^2+q^3tA+q^5t^2+q^4tA}{1-q^2}=\frac{(1+q)(q^4t^2+q^3tA)}{1-q^2}=\frac{q^4t^2}{1-q}+\frac{q^5t^3}{(1-q)^2}.
\end{equation}
Finally,
$$
(1-q)P_{000000}=q^6+q^5A+q^4B+q^4C=q^6+\frac{q^7t^2}{1-q}+q^7t+\frac{q^8t^2}{1-q}+\frac{q^8t^2}{1-q}+\frac{q^9t^3}{(1-q)^2}
$$
$$=
\frac{q^6}{(1-q)^2}\left(q^3t^3 - 2q^3t^2 + q^3t + q^2t^2 - 2q^2t + qt^2 + q^2 + qt - 2q + 1\right).
$$
Or, in a positive form 
\begin{equation*}
\frac{P_{000000}}{q^6}=\frac{1+qt}{1-q}+\frac{qt^2+2q^2t^2}{(1-q)^2}+\frac{q^3t^3}{(1-q)^3}.
\end{equation*}

Observe the $q, t$-symmetry of 
\begin{align*}
(1-q)^2c_{3,3}(q,t)&=t^3 q^{-6}(1-q)^3 P_{000000}(q, t^{-1})\\
&=q^3t^2+q^2t^3 - 2q^3t-2qt^3 + q^3+t^3  +q^2t+qt^2- 2q^2t^2 +qt
\end{align*}

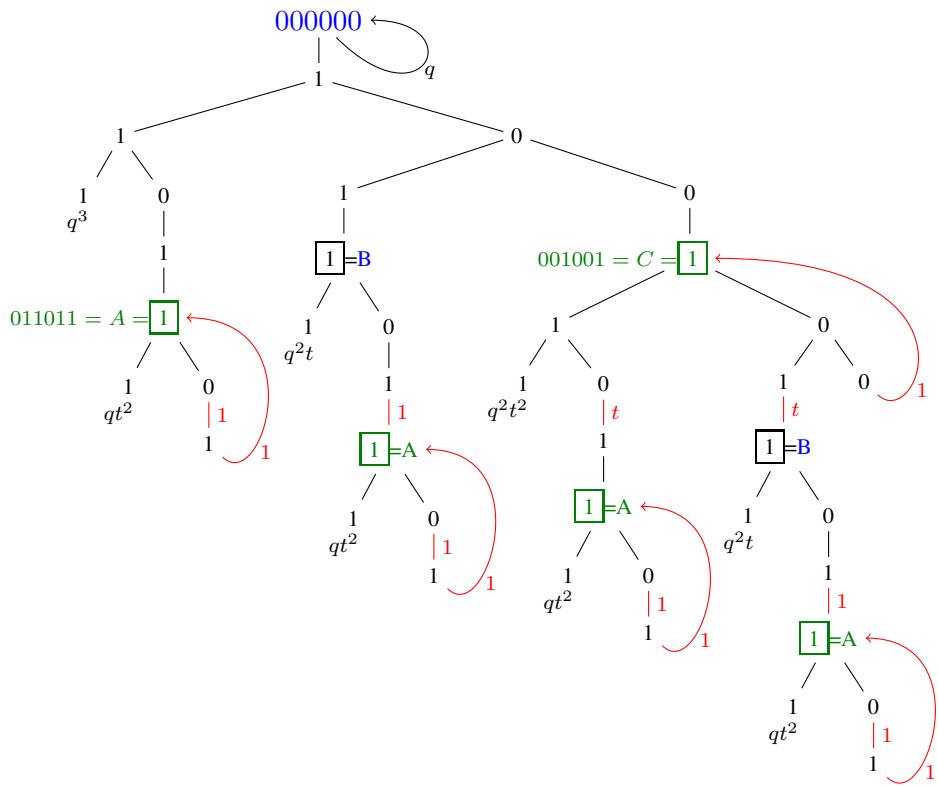
\begin{figure}
{\scalefont{0.7}
\begin{forest}
for tree={s sep= 7mm, l=0.5mm}
[{\scalefont{1.3}${\color{blue}000000}$}, name=root
[1 [1  [1\\{$q^3$}, align=right, base=top, name=000111]
       [0 [1 [\color{mygreen}{\fbox{1}}, name=A
[1\\{$qt^2$}, align=right,base=top, name=110111]
[0 [1, name=101101,edge=red, edge label={node[midway,right, red]{$1$}}]]]
]]]
  [0 [1 [{\fbox{1}=\color{blue}{B}}, name=B, 
	[1\\{$q^2t$}, align=right,base=top, name=010111]
	[0 [1 [\color{mygreen}{\fbox{1}={A}},name=A2,edge=red, edge label={node[midway,right, red]{$1$}}
[1\\{$qt^2$}, align=right,base=top, name=110111-2]
[0 [1, name=101101-2,edge=red, edge label={node[midway,right, red]{$1$}}]]] ]
] ]] 
     [0 [\color{mygreen}{ \fbox{1}}, name=C
[1 [1\\{$q^2t^2$}, align=right,base=top, name=100111]
  [0,[1, name=001101,edge=red, edge label={node[midway,right, red]{$t$}}
  [\color{mygreen}{\fbox{1}={A}},name=A3
[1\\{$qt^2$}, align=right,base=top, name=110111-3]
[0 [1, name=101101-3,edge=red, edge label={node[midway,right, red]{$1$}}]]]
]]]
  [0 [1 [{\fbox{1}=\color{blue}{B}}, name=B2, edge=red, edge label={node[midway,right, red]{$t$}}
	[1\\{$q^2t$}, align=right,base=top, name=010111-2]
	[0 [1 [\color{mygreen}{\fbox{1}={A}},name=A4,edge=red, edge label={node[midway,right, red]{$1$}}
[1\\{$qt^2$}, align=right,base=top, name=110111-4]
[0 [1, name=101101-4,edge=red, edge label={node[midway,right, red]{$1$}}]]] ]
]]]
	[0, name=100100]]]
]] ]]
\draw[->] (root) .. controls +(south east:2cm) and +(east:2cm) .. node[near start,right]{\quad $q$} (root);
\node at (A) [left] { ${\color{mygreen}{011011=A}=}$\; };
\draw[->, red] (101101) .. controls +(south east:1cm) and +(east:2cm) ..  node[near start,right, red]{$1$} (A);
\draw[->, red] (101101-2) .. controls +(south east:1cm) and +(east:2cm) ..  node[near start,right, red]{$1$} (A2);
\draw[->, red] (101101-3) .. controls +(south east:1cm) and +(east:2cm) ..  node[near start,right, red]{$1$} (A3);
\draw[->, red] (101101-4) .. controls +(south east:1cm) and +(east:2cm) ..  node[near start,right, red]{$1$} (A4);
\draw[->, red] (100100) .. controls +(south east:1cm) and +(east:4cm) ..  node[near start,right, red]{$1$} (C);
\node at (C) [left]{$\color{mygreen}{001001=C = } \,$};
\end{forest}
}
\caption{Compact decision tree for $(M,N)=(3,3)$ . 
}
\label{fig: 33 total}
\end{figure}

\end{example}

\begin{example}
\label{ex-46}
The decision tree  for $(M,N)=(4,6)$ is shown in
Figure \ref{fig: 46}.
We will use the shorthand notations $P_{0^{10}} := P_{0000000000}$
and $P_{(01)^5} = P_{0101010101}$ and so on. 
We compute
$$
P_{(01)^5} = q P_{(01)^5} + q^2 t (q^3 t + q^4 t^6)
$$
and so
$$
P_{(01)^5} =  \frac{q^5 t^6}{1-q} (1 + q t).
$$
Now 
$$
P_{0^{10}}=qP_{0^{10}}+qP_{0^9 1},
$$
hence
\begin{align*}
&P_{0^{10}}&=&
\frac{1}{1-q}\left(
  (q^9+t q^{12})P_{(01)^5}
+q^{17}t^7
+ q^{16}(t^6+t^7)
+ q^{15}(t^5+t^6)
+ q^{14}(t^4 + t^5 +t^6) \ +
 \right.&
\\
& & &\left.
q^{13}(t^3+ t^4+ 2t^5)
+ q^{12}(t^2+t^3+2t^4)
+ q^{11}(t+t^2+t^3)
+q^{10}
\right)&
\\
&&=& 
\frac{1}{(1-q)^2}\left(
q^{18}t^8 - q^{18}t^7 + q^{17}t^7 - q^{17}t^6 + q^{16}t^7
- q^{16}t^5 + q^{15}t^7 - q^{15}t^4 + 2q^{14}t^6  \ +
\right.&
\\
& &&
\left.
2q^{14}t^6 - q^{14}t^3 - q^{14}t^5 + 2q^{13}t^5 - q^{13}t^4 - q^{13}t^2
+ 2q^{12}t^4 - q^{12}t \ + 
\right.&
\\
& &&
\left.
q^{11}t^3 + q^{11}t^2 + q^{11}t
- q^{11} + q^{10}\right)&
\end{align*}
 Observe the $q, t$-symmetry of 
\begin{align*}
 (1-q)c_{4,6}(q,t)&={t^8(1-q)^2}{q^{-10}} P_{0^{10}}(q, t^{-1})\\
  &= -q^8t  - qt^8 + q^8+ t^8 - q^7t^2 - q^2t^7+ q^7t+ qt^7 
- q^6t^3 - q^3t^6+ q^6t  + qt^6 \ +
\\
&\ \ \ \ \ \  q^5t + qt^5 - q^5t^4 - q^4t^5
- q^4t^3 - q^3t^4 + 2q^4t^2 + 2q^2t^4 + 2q^3t^3.
\end{align*}
Also, as before, one gets a positive form:
\begin{equation*}
\begin{aligned}
\frac{P_{0^{10}}}{q^{10}}=\frac{1}{1-q}
\left(q^7t^7
\right.
+ &q^6t^6+q^6t^7
+ q^5t^5+q^5t^6
+ q^4t^4 + q^4t^5 +q^4t^6 \ +
\\
q^3t^3&+ q^3t^4+ 2q^3t^5
\left.
+ q^2t^2+q^2t^3+2q^2t^4
+ qt+qt^2+qt^3
+1 \right) \quad +\\
\frac{1}{(1-q)^2}(&q^4t^6+q^5t^7+q^7t^7+q^8t^8).
\end{aligned}
\end{equation*}

\end{example}

\begin{figure}
{\scalefont{0.5}

\begin{forest}
for tree={s sep= 4mm, l=0.5mm}
[{\scalefont{1.5}${\color{blue}0000000000}$}, name=root
[1
[1 [1  [1, name=4q6 ] [0, name=4q7t3] ]	 [0 [1, name=4q7t2] [0,name=4q8t4] ] ]
[0, name=22
[1, name=33 [1 [1 [1,name=6q5t] [0,name=6q6t4]]]
	[0, name=46 [1
		[1 [1 [1,name=8q4t3] [0,name=8q5t5]]]
		[0 [\color{mygreen}{\fbox{1}}, name=A [1 [1 [1, name=10q3t5] [0, name=10q4t6]]]
			[0 [\color{blue}{\fbox{1}}, name=P015 [1 [1, name=11red,edge=red,
 edge label={node[midway,right, red]{$t$}} 
				[1, name=12q3t5] [0,name=12q4t6]]]
				[0,name=P105]]]]
		 ]]] ]  
   [0, name=C
[1, name=47 [1 [1, name=6q6t2] [0, name=6q7t4]]]
[0, name=48 [1 [1 [1 [1,name=8q5t3] [0, name=8q6t5] ]]
	    	[0, name=68 [1 [1 [1 [1, name=10q4t4] [0, name=10q5t6]]]
				[0, name=88 [1
[1 [1, name=112red,edge=red, edge label={node[midway,right, red]{$t^2$}} 
	[1, name=12q4t3] [0,name=12q5t5]]]
[0 [\color{mygreen}{\fbox{1}}, name=A2 ,edge=red, edge label={node[midway,right, red]{$t$}} 
		[1 [1 [1, name=14q3t5] [0, name=14q4t6]]]
			[0 [\color{blue}{\fbox{1}}, name=P0152 [1 [1, name=15red,edge=red,
 edge label={node[midway,right, red]{t}} 
				[1, name=16q3t5] [0,name=16q4t6]]]
				[0,name=P1052]]]] ]
]] ]
]]]]]]
] 
\node at (A) [left] { ${\color{mygreen} A=}$\; };
\node at (A2) [left]  { ${\color{mygreen} A=}$\; };
{\scalefont{1.4}
\draw[->, red] (P105) .. controls +(south east:1cm) and +(east:2cm) .. 
node[near start,right, red]{$\;1$} (P015);
\draw[->, red] (P1052) .. controls +(south east:1cm) and +(east:2cm) .. 
node[near start,right, red]{$\;1$} (P0152);
\draw[->] (root) .. controls +(south east:2cm) and +(east:2cm) .. 
node[near start,right]{\quad\; $q$} (root);
\node at (P015) [left] { ${\color{blue} (01)^5=}$\;};
\node at (P0152) [left] { ${\color{blue} (01)^5=}$\;};
\node at (4q6) [below] {\, $q^6$};
\node at (4q7t3) [below] {\, $q^7 t^3$};
\node at (4q7t2) [below] {\, $q^7 t^2$};
\node at (4q8t4) [below] {\, $q^8 t^4$};
\node at (6q5t) [below] {\, $q^5 t$};
\node at (6q6t4) [below] {\, $q^6 t^4$};
\node at (6q6t2) [below] {\, $q^6 t^2$};
\node at (6q7t4) [below] {\, $q^7 t^4$};
\node at (8q4t3) [below] {\, $q^4 t^3$};
\node at (8q5t5) [below] {\, $q^5 t^5$};
\node at (8q5t3) [below] {\, $q^5 t^3$};
\node at (8q6t5) [below] {\, $q^6 t^5$};
\node at (10q3t5) [below] {\, $q^3 t^5$};
\node at (10q4t6) [below] {\, $q^4 t^6$};
\node at (10q4t4) [below] {\, $q^4 t^4$};
\node at (10q5t6) [below] {\, $q^5 t^6$};
\node at (12q3t5) [below] {\, $q^3 t^5$};
\node at (12q4t6) [below] {\, $q^4 t^6$};
\node at (12q4t3) [below] {\, $q^4 t^3$};
\node at (12q5t5) [below] {\, $q^5 t^5$};
\node at (14q3t5) [below] {\, $q^3 t^5$};
\node at (14q4t6) [below] {\, $q^4 t^6$};
\node at (16q3t5) [below] {\, $q^3 t^5$};
\node at (16q4t6) [below] {\, $q^4 t^6$};
} 
\end{forest}


} 
\caption{Compact decision tree for $(M,N)=(4,6).$}
\label{fig: 46}
\end{figure}
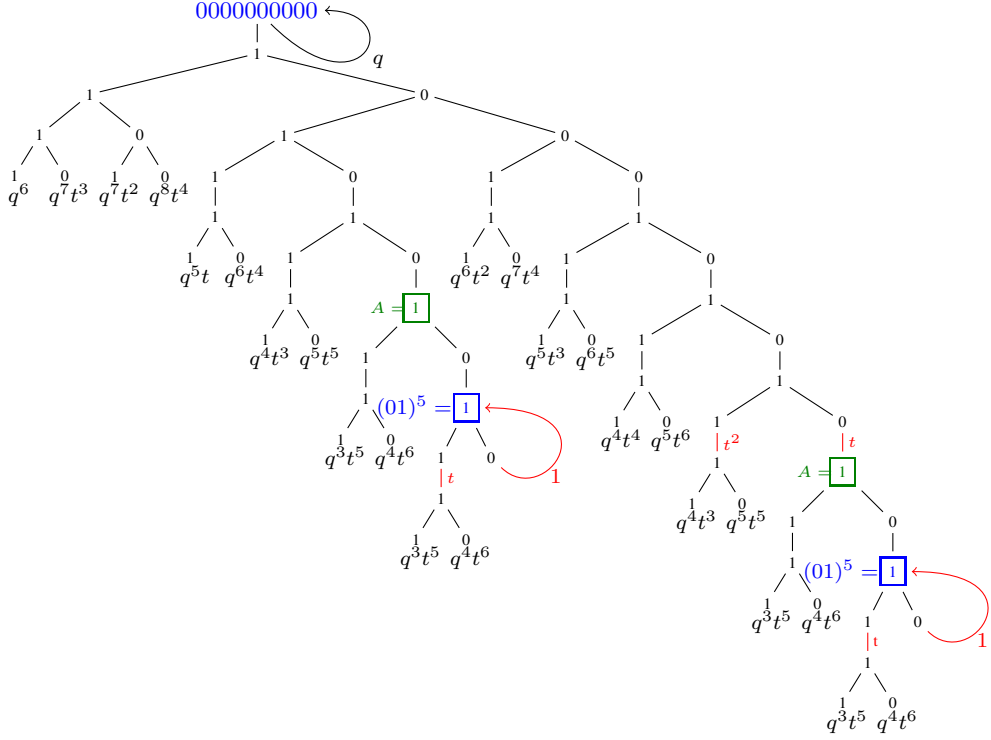

\section{Comparison with the work of Hogancamp and Mellit.}
\label{sec:comparison}

\subsection{Hogancamp-Mellit recursion}

Our next goal is match this recursion with the $a=0$ specialization of the following recursion due to Hogancamp and Mellit \cite{HogMellit}. 

\begin{definition}[\cite{HogMellit}]
\label{def:R}
The power series $R_{\bx,\by}(q,t,a)$ in variables $q,t$ and $a$ depend on a pair of words $\bx$ and $\by$ in the alphabet $\{0,\times\}.$ These power series satisfy the following recursive relations:
\begin{equation*}
\begin{aligned}
&R_{0{\bx},0{\by}}=t^{-|{\bx}|}R_{{\bx}\times,{\by}\times}+qt^{-|{\bx}|}R_{{\bx}0,{\by}0},\\
&R_{\times{\bx},0{\by}}=R_{{\bx}\times,{\by}},\\
&R_{0{\bx},\times{\by}}=R_{{\bx},{\by}\times},\\
&R_{\times{\bx},\times{\by}}=(t^{|{\bx}|}+a)R_{{\bx},{\by}},\\
&R_{\emptyset,\emptyset}=1,
\end{aligned} 
\end{equation*}
where $|{\bx}|$ denotes the number of $\times$'s in ${\bx}.$
\end{definition}

\begin{remark}
\label{rem:reverse}
Our recursion differs from the one in \cite{HogMellit} by reversing the order in both sequences $\bx,\by$. 
\end{remark}

In order to do so, we will need to go through certain reformulations and also adjust both the $\area$ and $\codinv$ statistics. First, we will need to replace the binary sequence $\bu$ of length $N+M$ by two sequences $(\bv,\bw)$ in the alphabet $\{0,\bullet,\times\}$ of lengths $M$ and $N$ respectively. Sequence $\bv$ records gaps (encoded by $0$), $N$-generators (encoded by $\times$), and the rest of the elements of $\Delta$ (encoded by $\bullet$) on the interval $[N,N+M-1].$ Similarly, sequence $\bw$ records gaps, $M$-generators, and the rest of the elements of $\Delta$ on the interval $[M,N+M-1].$ More formally, one gets the following definition:

\begin{definition}\label{Definition: v,w from u}
Let ${\bu}=(u_0,\ldots,u_{N+M-1})\in \{0,1\}^{N+M}$ be an admissible binary sequence. Define ${\bv}=(v_0,\ldots,v_{M-1})\in \{0,\bullet,\times\}^M$ as follows:
\begin{enumerate}
\item $v_i=0$ whenever $u_{N+i}=0,$
\item $v_i=\bullet$ whenever $u_{N+i}=u_i=1,$ and
\item $v_i=\times$ whenever $u_{N+i}=1$ and $u_i=0.$
\end{enumerate}
Similarly, define the sequence ${\bw}=(w_0,\ldots,w_{N-1})\in \{0,\bullet,\times\}^N$ as:
\begin{enumerate}
\item $w_i=0$ whenever $u_{M+i}=0,$
\item $w_i=\bullet$ whenever $u_{M+i}=u_i=1,$ and
\item $w_i=\times$ whenever $u_{M+i}=1$ and $u_i=0.$
\end{enumerate}
We say that a pair of sequences $(\bv,\bw)$ is
admissible if $\bv$ and $\bw$ are obtained from an admissible binary
sequence $\bu$ according to the rule above.
\end{definition} 
Clearly, the pair of sequences $\bv,\bw$ fully determine the binary
sequence $\bu.$
In other words, the above Definition \ref{Definition: v,w from u}
 describes a map
$${\mathbf{b}}:\{0,1\}^{M+N} 
\to \{0,\bullet,\times\}^M \times \{0,\bullet,\times\}^N
$$
and it is injective when restricted to the domain of
admissible sequences. Indeed, 
$$
\bu_i=\begin{cases}
1\ \text{if}\ i\in [0,M-1]\ \text{and}\ \bv_i=\bullet,\\
1\ \text{if}\ i\in [M,M+N-1]\ \text{and}\ \bw_{i-M}\in \{\times,\bullet\},\\
0\ \text{otherwise}.\\
\end{cases}
$$


\begin{remark}
\label{rem: lambda via v and w}
If $(\bv,\bw)=\mathbf{b}(\bu)$ then by Remark \ref{rem: lambda} we have $|\bv|=|\bw| = \lambda(\bu)$.
Here, as above, $|\bv|$ denotes the number of $\times$ entries in $\bv$.
\end{remark}
\begin{definition}\label{Definition: Iv,w Pv,w}
By slightly abusing notation, we set
\begin{equation*}
I_{{\bv},{\bw}}:=I_{\bu},
\end{equation*}
and
\begin{equation*}
P_{{\bv},{\bw}}(q,t):=P_{\bu}(q,t),
\end{equation*}
\end{definition} 


Now we can reformulate the recursion from Theorem \ref{Theorem: recursion binary N+M} using the new notation.

\begin{proposition}\label{Theorem: recursion P_v,w}
The recursion \eqref{eq-recursion binary N+M} is equivalent to the following recursion:
\begin{equation*}
\begin{aligned}
&P_{0{\bv},0{\bw}}=q(P_{{\bv}\times,{\bw}\times}+P_{{\bv}0,{\bw}0}),\\
&P_{\times{\bv},0{\bw}}=qP_{{\bv}\times,{\bw}\bullet},\\
&P_{0{\bv},\times{\bw}}=qP_{{\bv}\bullet,{\bw}\times},\\
&P_{\times{\bv},\times{\bw}}=qP_{{\bv}\bullet,{\bw}\bullet},\\
&P_{\bullet{\bv},\bullet{\bw}}=t^{|{\bv}|}P_{{\bv}\bullet,{\bw}\bullet}.\\
\end{aligned} 
\end{equation*}
\end{proposition}

This recursion looks very similar to the $a=0$ version of the 
Hogancamp-Mellit \cite{HogMellit} recursion, but not exactly the same. In order to get an exact match, let us make the following adjustments to the statistics:

\begin{definition}
Define
statistics $\area'$ and $\codinv'$ on the set $\IMN$ of $(M,N)$-invariant subsets in $\BZ_{\ge 0}$ by 
\begin{equation*}
\area'(\Delta)=\sharp (\overline{\Delta}\cap\BZ_{\ge N+M}),
\end{equation*}
and
\begin{equation*}
\begin{aligned}
\codinv'(\Delta)=\sum\limits_{\a\in\Ngen(\Delta)}&\sharp\left([\a,\a+M-1]\cap \overline{\Delta}\cap \BZ_{\ge N+M}\right)
-\frac{\lambda(\Delta)(\lambda(\Delta)-1)}{2}\\
\end{aligned}
\end{equation*}
where as in Definition \ref{def:lambda}
above $\lambda(\Delta)=\sharp(\Ngen(\Delta)\cap [N,N+M-1]).$ 
\end{definition}






\begin{definition}
As before, let ${\bu}=(u_0,\ldots,u_{N+M-1})\in\{0,1\}^{N+M}$ be an admissible binary sequence. The generating series $Q_{\bu}(q,t)$ is defined by
\begin{equation*}
Q_{\bu}(q,t):=\sum\limits_{\Delta\in I_{\bu}} t^{-\codinv'(\Delta)}q^{\area'(\Delta)}.
\end{equation*}
We also set
\begin{equation*}
Q_{{\bv},{\bw}}(q,t):=Q_{\bu}(q,t),
\end{equation*}
where the sequences ${\bv}=(v_0,\ldots,v_{M-1})\in\{0,\bullet,\times\}^M$ and ${\bw}=(w_0,\ldots,w_{N-1})\in\{0,\bullet,\times\}^N$ are determined in the same way as in Definition \ref{Definition: v,w from u},
i.e. $(\bv,\bw) = \mathbf{b}(\bu)$.
\end{definition}

Note that for any $\Delta\in I_{\zero}$ one gets $\area'(\Delta)
=
-N-M + \area(\Delta)$ and $\codinv'(\Delta)=\codinv(\Delta).$ Therefore,
\begin{equation}\label{Equation: Q vs P}
Q_{\zeroMN}(q,t)=q^{-N-M}P_{\zeroMN}(q,t^{-1}).
\end{equation}

\begin{remark}
\label{rem: comparison 1s in the end}
More generally, let $\Delta\in I_{0^{M+N-k}1^k}$ for $k\le \min(M,N)$. Then 
$$
\area'(\Delta)=-N-M+k+\area(\Delta)\ \text{and}\ \codinv'(\Delta)=\codinv(\Delta)-\frac{k(k-1)}{2}.
$$
and
\begin{equation}
\label{eq: Q vs P with ones at the end}
Q_{0^{M-k}\times^k,0^{N-k}\times^k}(q,t)=q^{-N-M+k}t^{k(k-1)/2}P_{0^{M+N-k}1^k}(q,t^{-1}).
\end{equation}
\end{remark}

\begin{theorem}\label{Theorem: recursion Q_v,w}
The following recursion holds:
\begin{equation*}
\begin{aligned}
&Q_{0{\bv},0{\bw}}=t^{-|{\bv}|}Q_{{\bv}\times,{\bw}\times}+qt^{-|{\bv}|}Q_{{\bv}0,{\bw}0},\\
&Q_{\times{\bv},0{\bw}}=Q_{{\bv}\times,{\bw}\bullet},\\
&Q_{0{\bv},\times{\bw}}=Q_{{\bv}\bullet,{\bw}\times},\\
&Q_{\times{\bv},\times{\bw}}=t^{|{\bv}|}Q_{{\bv}\bullet,{\bw}\bullet},\\
&Q_{\bullet{\bv},\bullet{\bw}}=Q_{{\bv}\bullet,{\bw}\bullet},\\
\end{aligned} 
\end{equation*}
\end{theorem}
\begin{proof}
Let $\Delta\in \IMN$ be an invariant subset.
Similar to the proof of Theorem \ref{Theorem: recursion binary N+M}, we consider the shift map $\rho.$

For the $\area'$ statistic we get that if $N+M\notin\Delta,$ then 
\begin{equation*}
\area'(\rho(\Delta))=\area'(\Delta)-1,
\end{equation*}
while if $N+M\in\Delta,$ then 
\begin{equation*}
\area'(\rho(\Delta))=\area'(\Delta).
\end{equation*}

One has to consider the two summands of the $\codinv'$ statistic separately. The first summand is given by 
\begin{equation}
\label{eq: first summand}
\sum\limits_{\a\in\Ngen(\Delta)}\sharp([\a,\a+M-1]\cap \overline{\Delta}\cap\BZ_{\ge N+M}).
\end{equation}
Clearly, it is enough to sum over the $N$-generators that are bigger than $N,$ because otherwise the interval $[\a,\a+M-1]$ does not intersect $\BZ_{\ge N+M}.$ We can thus rewrite \eqref{eq: first summand} as
\begin{equation*}
\sum\limits_{\a\in\Ngen(\Delta)\cap\BZ_{>N}}\sharp([\a,\a+M-1]\cap \overline{\Delta}\cap\BZ_{\ge N+M}).
\end{equation*}
The $N$-generators of $\rho(\Delta)$ agree with those of $\Delta$
shifted down by $1$, except in the case that $0 \in \NgenD$, 
in which case it is replaced by $\Ngen(\rho(\Delta)) \ni N-1 < N$; so regardless
$\Ngen(\rho(\Delta))\cap\BZ_{>N}=(\Ngen(\Delta)\cap\BZ_{>N+1})-1$ 
holds.
Hence 
\begin{equation*}
\begin{aligned}
\sum\limits_{\a\in\Ngen(\rho(\Delta))}&\sharp([\a,\a+M-1]\cap \overline{\rho(\Delta)}\cap\BZ_{\ge N+M})\\
=&\sum\limits_{\a\in\Ngen(\Delta)}\sharp([\a,\a+M-1]\cap \overline{\Delta}\cap\BZ_{\ge N+M+1}).
\end{aligned}
\end{equation*} 
Therefore if $N+M\in\Delta$ then \eqref{eq: first summand} does not change. If $N+M\notin\Delta$ then
\begin{equation*}
\begin{aligned}
\sum\limits_{\a\in\Ngen(\rho(\Delta))}\sharp([\a,&\a+M-1]\cap \overline{\rho(\Delta)}\cap\BZ_{\ge N+M})\\
=\sum\limits_{\a\in\Ngen(\Delta)}&\sharp([\a,\a+M-1]\cap \overline{\Delta}\cap\BZ_{\ge N+M+1})\\
=\sum\limits_{\a\in\Ngen(\Delta)}&\sharp([\a,\a+M-1]\cap \overline{\Delta}\cap\BZ_{\ge N+M})-\sum\limits_{\a\in\Ngen(\Delta)}&\sharp([\a,\a+M-1]\cap\{N+M\})\\
=\sum\limits_{\a\in\Ngen(\Delta)}&\sharp([\a,\a+M-1]\cap \overline{\Delta}\cap\BZ_{\ge N+M})-\lambda(\Delta)\\
\end{aligned}
\end{equation*} 
since $N \not\in \Ngen(\Delta)$ in this case.
Recall that if $\Delta\in I_{\bv,\bw},$ then by Remark \ref{rem: lambda via v and w} one has $\lambda(\Delta)=|\bv|=|\bw|$.
Combining the above in comparing $\Delta$ to $\rho(\Delta)$, we get:
\begin{enumerate}
\item If $\Delta\in I_{0\bv,0\bw},$ then $0,N,M\notin\Delta.$ Then we can either have $N+M\in\Delta$ or $N+M\notin \Delta.$
If $N+M\in\Delta,$ then $\rho(\Delta)\in I_{\bv \times, \bw \times}$ and
 $\area'$ and \eqref{eq: first summand} do not change, but $\lambda(\Delta)$ increases by 1, i.e., $\lambda(\rho(\Delta))=\lambda(\Delta)+1$, which decreases the second summand of the $\codinv'$ statistic by
$\lambda(\Delta)=|\bv|$.
If $N+M\notin\Delta,$ then $\rho(\Delta)\in I_{\bv 0, \bw 0}$ and   
\begin{equation*}
\area'(\rho(\Delta))=\area'(\Delta)-1;
\end{equation*}
furthermore \eqref{eq: first summand} decreases by $\lambda(\Delta)=|\bv|$,
and $\lambda(\Delta) = \lambda(\rho(\Delta))$ does not change. 
All combined, we get the first relation:
\begin{equation*}
Q_{0{\bv},0{\bw}}=t^{-|{\bv}|}Q_{{\bv}\times,{\bw}\times}+qt^{-|{\bv}|}Q_{{\bv}0,{\bw}0}.
\end{equation*}

\item If $\Delta\in I_{\times\bv,0\bw},$ then $0,M\notin\Delta,$ but $N\in\Ngen(\Delta).$ Then also $N+M\in\Delta$
with $N+M \in \Ngen(\Delta)$ but $N+M \not\in \Mgen(\Delta)$ so that
$\rho(\Delta) \in I_{\bv \times, \bw \bullet}$.
 Furthermore $\area'$ and \eqref{eq: first summand} do not change, 
and $\lambda(\Delta)=|\times\bv|=|\bv\times| =\lambda(\rho(\Delta))$
does not change either.

The case of $\Delta\in I_{0\bv,\times\bw}$ is analogous to the above. We get

\begin{equation*}
\begin{aligned}
&Q_{\times{\bv},0{\bw}}=Q_{{\bv}\times,{\bw}\bullet},\\
&Q_{0{\bv},\times{\bw}}=Q_{{\bv}\bullet,{\bw}\times}.
\end{aligned} 
\end{equation*}

\item If $\Delta\in I_{\times\bv,\times\bw},$ then $0\notin\Delta,$ but $N\in\Ngen(\Delta)$ and $M\in\Mgen(\Delta).$ Then also $N+M\in\Delta,$ therefore $\area'$ and \eqref{eq: first summand} do not change, and $\lambda(\Delta)$ decreases by 1.
Since $N+M$ is neither an $N$-generator nor $M$-generator,
$\rho(\Delta) \in I_{{\bv}\bullet,{\bw}\bullet}$.  We get
\begin{equation*}
Q_{\times{\bv},\times{\bw}}=t^{|{\bv}|}Q_{{\bv}\bullet,{\bw}\bullet}.
\end{equation*}

\item Finally, if $\Delta\in I_{\bullet\bv,\bullet\bw},$ then $0,N,M\in\Delta,$ and also $N+M\in\Delta,$ but $N+M \notin \NgenD$ and $N+M \notin \MgenD$. 
 Therefore, no statistics change in this case and we get

\begin{equation*}
Q_{\bullet{\bv},\bullet{\bw}}=Q_{{\bv}\bullet,{\bw}\bullet}.
\end{equation*}
\end{enumerate}
\end{proof}

The final observation is that in the recursion for $Q$ one can completely ignore all the $\bullet$'s, stated more precisely in Theorem
\ref{thm:ignore bullets}.
This motivates the definition of the following partial order
that we will use in the proof of our next result.

\begin{definition}
Let $(\bv,\bw)$ and $(\bv',\bw')$ be two admissible pairs of sequences. We say that $(\bv',\bw')\prec(\bv,\bw)$ if and only if either the total number of $\bullet$'s in the sequences $\bv'$ and $\bw'$ is greater than the total number of $\bullet$'s in the sequences $\bv$ and $\bw,$ or those numbers are the same, but then $|\bv'|>|\bv|,$ i.e. the number of $\times$ in $\bv'$ is greater then the number of $\times$'s in $\bv.$
\end{definition}

Note that $(\bullet^M, \bullet^N)$ is minimal with respect to $\prec$.
Further if $(\bv,\bw)$ and $(\bv',\bw')$ have the same total
number of $\bullet$'s and $\times$'s, but $(\bv,\bw) \neq (\bv',\bw')$
then they are incomparable. 

\begin{definition}
\label{def:phi}
Let $\phi$ be the map from the words in the alphabet $\{0,\bullet,\times\}$ to the words in the alphabet $\{0,\times\}$ given by simply forgetting all $\bullet$s.
\end{definition}
\begin{theorem}
\label{thm:ignore bullets}
Let $\bx=\phi(\bv)$ and $\by=\phi(\bw).$ Then
\begin{equation*}
R_{\bx,\by}(q,t,0)=Q_{\bv,\bw}(q,t).
\end{equation*}
\end{theorem}

\begin{proof}
The proof is by induction with respect to the
order $\prec$.

The base is covered by the normalization conditions:
\begin{equation*}
R_{\emptyset,\emptyset}(q,t,0)=1=Q_{\bullet\ldots\bullet,\bullet\ldots\bullet}.
\end{equation*}
Consider a pair of admissible sequences $(\bv,\bw)$ and assume that for all pairs $(\bv',\bw')$ such that $(\bv',\bw')\prec(\bv,\bw)$  we have already proved that 
\begin{equation*}
R_{\phi(\bv'),\phi(\bw')}(q,t,0)=Q_{{\bv'},{\bw'}}(q,t).
\end{equation*}
Now we apply the recurrence relations to $Q_{\bv,\bw}.$ Note that all the recurrence relations
listed in Theorem \ref{Theorem: recursion Q_v,w},
except the first one and the last one increase the number of $\bullet$'s, in which case we are done by the inductive hypothesis and by using the corresponding recurrence relation for $R_{\bx,\by}(q,t,0).$ Therefore, the only cases left are when either both $\bv$ and $\bw$ start with $0$'s, and we are forced to use the first relation:
\begin{equation*}
Q_{0{\bv},0{\bw}}=t^{-|{\bv}|}Q_{{\bv}\times,{\bw}\times}+qt^{-|{\bv}|}Q_{{\bv}0,{\bw}0},
\end{equation*}
or they both start with $\bullet$'s, and we are forced to use the last relation:
\begin{equation*}
Q_{\bullet{\bv},\bullet{\bw}}=Q_{{\bv}\bullet,{\bw}\bullet}.
\end{equation*}
Note that in the case of the first relation our only concern is the second summand, as for the first summand we have
$({\bv}\times,{\bw}\times)\prec(0\bv,0\bw).$
In either case, we keep applying the recurrence relations to the terms not covered by the inductive assumption until either no such terms 
are 
left, and we are done by the inductive assumption and the corresponding recurrence relations for $R_{\bx,\by}(q,t,0),$ or we get into a cycle and get back the same term $Q_{\bv,\bw}.$ In this case we get

\begin{equation*}
Q_{\bv,\bw}=\gamma Q_{\bv,\bw}+\sum\limits_{(\bv',\bw')\prec(\bv,\bw)} \gamma_{\bv',\bw'}Q_{\bv',\bw'},
\end{equation*}
where $\gamma$ and all $\gamma_{\bv',\bw'}$'s are monomials in $q$ and $t$ (or $0$s). Note that in this case both sequences $\bv$ and $\bw$ don't contain any $\times$'s, and neither do the intermediate pairs of sequences along the cycle, as those pairs are simply cyclic shifts of $(\bv,\bw).$ Also, since we are not in the base case of the induction, the period contains at least one $0,$ meaning that we used the first relation at least once. Therefore, the coefficient
$\gamma = q^k$
where $k$ is a positive integer. We can rewrite the above equation as
\begin{equation*}
Q_{\bv,\bw}=\frac{\sum\limits_{(\bv',\bw')\prec(\bv,\bw)} \gamma_{\bv',\bw'}Q_{\bv',\bw'}}{1-q^k}.
\end{equation*}
Now we are done by the inductive assumptions and the corresponding recurrence relations for $R_{\bx,\by}(q,t,0),$
 as in all previous cases.
\end{proof}

By Theorem \ref{thm:ignore bullets}, Lemma \ref{lem: zero and catalan} and Remark \ref{rem: comparison 1s in the end} we immediately get the following.

\begin{corollary}
\label{cor:R vs catalan}
We have
$$
R_{0^M,0^N}(q,t,0)=Q_{0^M,0^N}(q,t)=q^{-N-M}P_{0^{M+N}}(q,t^{-1})=\frac{t^{-\delta(N,M)}}{1-q}c_{M,N}(q,t).
$$
\end{corollary}

\begin{corollary}
\label{cor: R with ones}
For $k\le \min(M,N)$ we have
$$
R_{0^{M-k}\times^k,0^{N-k}\times^{k}}(q,t,0)=Q_{0^{M-k}\times^k,0^{N-k}\times^{k}}(q,t)=q^{-N-M+k}t^{k(k-1)/2}P_{0^{M+N-k}1^k}(q,t^{-1}).
$$
\end{corollary}



\subsection{Relation to Khovanov-Rozansky homology}
\label{sec:homology}

For the reader's convenience, in this section we give a short summary of the main results of \cite{HogMellit}
and provide a topological interpretation of the series $R_{\bx,\by}$ which appeared in their work. All further details
can be found in \cite{HogMellit}.

In \cite{KR2,KhSoer} Khovanov and Rozansky defined a new link invariant called HOMFLY-PT (or Khovanov-Rozansky) homology. To each link $L$ they associate a triply graded vector space $\mathcal{H}(L)=\oplus_{i,j,k}\mathcal{H}^{i,j,k}(L)$.
It is usually infinite dimensional, but all graded components $\mathcal{H}^{i,j,k}(L)$ are finite dimensional. The generating function for  their dimensions is usually called Poincar\'e series:
$$
\mathcal{P}(Q,T,A)=\sum_{i,j,k}Q^iA^jT^k\dim \mathcal{H}^{i,j,k}(L).
$$
The Euler characteristic of this homology 
$
\mathcal{P}(Q,-1,A)=\sum_{i,j,k}Q^iA^j(-1)^k\dim \mathcal{H}^{i,j,k}(L)
$
equals the HOMFLY-PT polynomial of $L$.
For various reasons it is useful to make a change of variables
$$
q=Q^2,\ t=T^2Q^{-2},\ a=AQ^{-2}.
$$
Recall that the $(M,N)$ torus link has $d=\gcd(M,N)$ components, each of which is a $(m,n)$ torus knot but they are linked nontrivially. For example, $(2,2)$ torus link is a pair of unknots linked the simplest possible way. 

\begin{theorem}(\cite{HogMellit})
The Poincar\'e series of the HOMFLY-PT homology of the $(M,N)$ torus link equals
$R_{0^M,0^N}(q,t,a)$.
\end{theorem}

\begin{theorem}(\cite{HogMellit})
\label{thm:knots}
Let $M=dm, N=dn$ where $\gcd(m,n)=1$ and $m,n,d \in \ZZ$.
Consider the $(m,n)$ torus knot colored by the representation $\mathrm{Sym}^d$. Then the Poincar\'e series of the corresponding colored homology equals $$\prod_{i=1}^{d}\frac{1}{1-qt^{i-d}}R_{0^{M-d}1^d,0^{N-d}1^d}(q,t,a).$$
\end{theorem}

Corollaries \ref{cor:R vs catalan} and \ref{cor: R with ones} immediately give a combinatorial interpretation of these results. 
Indeed, Corollary \ref{cor:R vs catalan} relates  the HOMFLY-PT homology  (at $a=0$)  of the $(M,N)$ torus link to the rational Catalan polynomial $c_{M,N}(q,t)$ while Corollary \ref{cor: R with ones} relates the colored homology of the $(m,n)$ torus knot to the polynomial $P_{0^{M+N-d}1^d}(q,t)$. 

\begin{example}
The  $S^2$-colored homology of the trefoil corresponds to $d=2,m=2,n=3$, so $M=4,N=6$.
The corresponding Poincar\'e series corresponds to the polynomial $P_{0^{8}1^2}(q,t)$
which can be easily computed from Figure \ref{fig: 46}.
\end{example}

\begin{remark}
It is easy to see that the $(M,N)$-invariant subsets in $I_{0^{M+N-d}1^d}$ are in bijection with $d$-tuples 
of $(m,n)$-invariant subsets. Therefore the polynomials $P_{0^{M+N-d}1^d}(q,t)$ and $R_{0^{M-d}1^d,0^{N-d}1^d}(q,t,0)$ have $c_{m,n}(1,1)^d$ terms in agreement with the ``exponential growth conjecture'' of \cite{GS,GGS}. 
\end{remark}

\subsection{Denominators of the rational functions}
\label{sec-denoms}

The Hogancamp-Mellit recursion for the series
$R_{\bx,\by}(q,t,0)$ as well as for the full series
$R_{\bx,\by}(q,t,a)$
has an important advantage over the recursion for $P_{\bu}(q,t):$ one can use it to show that the denominator of $R_{\bx,\by}(q,t,a)$ can be simplified to a power of $(1-q)$ as compared to 
the denominator $\prod_{i=1}^d (1-q^{\ell_i}),$ $0<\ell_i<d$
predicted by Theorem \ref{Theorem: recursion unique} before 
reducing expressions.

Indeed, let us apply the same argument as in the proof of Theorems \ref{Theorem: recursion P_v,w} and \ref{Theorem: recursion Q_v,w} to the recursion for $R_{\bx,\by}(q,t,a).$ Similar to Theorem \ref{Theorem: recursion Q_v,w}, let us introduce a partial order on pairs $(\bx,\by)$ of words in the alphabet $\{0,\times\}$
which by abuse of notation we also call $\prec.$

\begin{definition}
Let $(\bx,\by)$ and $(\bx',\by')$ be two pairs of words in the alphabet $\{0,\times\}$. We say that $(\bx',\by')\prec(\bx,\by)$ if and only if either the sum of lengths of $\bx'$ and $\by'$ is less than the sum of lengths of $\bx$ and $\by,$ or those sums are the same, but then $|\bx'|>|\bx|,$ i.e. the number of $\times$ in $\bx'$ is greater then the number of $\times$'s in $\bx.$
\end{definition}

Note that the only case of the recursion that doesn't go down in the above order is the second term of the first recursive relation, where the pair
$(0\bx,0\by)$ gets replaced by $(\bx 0,\by 0).$
Therefore, the only way one gets back the same pair of words one started from is if the initial pair was $(0\ldots 0,0\ldots 0),$ in which case one gets denominator $(1-q):$

\begin{equation*}
R_{00\ldots 0,00\ldots 0}=R_{0\ldots 0\times,0\ldots 0\times}+qR_{0\ldots 00,0\ldots 00},
\end{equation*}

or

\begin{equation*}
R_{00\ldots 0,00\ldots 0}=\frac{R_{0\ldots 0\times,0\ldots 0\times}}{1-q}.
\end{equation*}




In all other cases the recursion relations reduce to terms with pair of words smaller in our order without introducing any denominators.
This in particular holds for the fourth equation Definition
\ref{def:R} which is the only relation involving the parameter $a$.

\begin{corollary}
The power series $R_{\bx,\by}(q,t,a)$ can be expressed as a rational function with the denominator equal to a power of $(1-q).$ 
\end{corollary} 

Combining this with Theorem \ref{Theorem: recursion Q_v,w}, Equation (\ref{Equation: Q vs P}), and Theorem \ref{Theorem: recursion P_v,w}, we get

\begin{corollary}
The power series $P_{0^{M+N}}(q,t)$ can be expressed as a rational function with denominator equal to  $(1-q)^d,$ where $d=\gcd(N,M).$ 
In particular,
the power series $c_{M,N}(q,t)$ can be expressed as a rational function
with denominator equal to  $(1-q)^{d-1},$ where $d=\gcd(N,M).$
\end{corollary}

In terms of the decision trees, as we replace the labels $(\bv,\bw)$ by the labels $(\bx,\by)=(\phi(\bv),\phi(\bw))$ by forgetting $\bullet$s
(as in Definition \ref{def:phi}), more vertices become identical. As a result, long cycles in the decision tree for $Q_{\bv,\bw}$ locally trivially cover the length one cycles in the decision tree for $R_{\bx,\by},$ with all the branches matching up perfectly. Note that one has to use the decision tree for $Q_{\bv,\bw}$ rather then the decision tree for $P_{\bu},$ i.e., the weights of edges have to correspond to $\area'$ and $\codinv',$ rather than $\area$ and $\codinv.$ Also, the edges $(\bullet\bv,\bullet\bw)\to(\bv\bullet,\bw\bullet)$ become trivial and should be collapsed. We illustrate this in the following example:

\begin{example}
\label{ex:33HM}

In Figure \ref{fig: 33hogmelcolor} we show the decision tree for the Hogancamp-Mellit recursion for $(M,N)=(3,3).$ 
Since we start from $(\bv,\bw)=(000,000)$, it is easy to see that $\bv=\bw$ everywhere in the tree and so we may just record $\bv.$
In fact, for $M=N$ this is always the case, and the recursion in \cite{HogMellit} is identical to the one in \cite{EH}.
In the Figure \ref{fig: 33hogmelcolor} we label the vertices by $\phi(\bv)$.
We remind the reader of Remark \ref{rem:reverse}.

One can compare the decision trees in Figures \ref{fig: 33},\ref{fig: 33 sub},\ref{fig: 33 total} to the one in
Figure \ref{fig: 33hogmelcolor}.
The former surject to the latter via $\bu \sim \bv \mapsto \phi(\bv)$, and the marked nodes $A,B,C$ are mapped to the corresponding nodes. Indeed,
$$
A=011\underline{011}\sim 0\bullet\bullet \xmapsto{\phi} 0,\ B=001\underline{011}\sim 0\times\bullet\xmapsto{\phi}  0\times,
$$
$$
C=001\underline{001}\sim 00\bullet\xmapsto{\phi}  00.
$$
Here we underline the last three letters in $\bu$ which yield both $\bv$ and $\bw$.

One may verify from Figure \ref{fig: 33hogmelcolor} that
\begin{align*}
&R_{000,000}(q,t,0)& = &\frac 1{1-q}\left( 
1 + qt^{-1}+ (qt^{-2}+q^2t^{-2}) R_{0,0} + q^2 t^{-2} R_{00,00}
\right)&
\\
&& = & \frac{1+q t^{-1}}{1-q} + \frac{qt^{-2} + 2q^2 t^{-2}}{(1-q)^2}
	 +\frac{q^3t^{-3}}{(1-q)^3}.&
\end{align*}
In this decision tree, edges from the first relation in
Definition \ref{def:R} are colored \textcolor{blue}{blue},
while those involving the fourth relation are black and
all black edges are weighted with
$t^{|\bx|}+a$ 
(or with non-negative powers of $t$
in the $a=0$ specialization).
Compare this to $R_{000,000}(q,t,a)$ in Example \ref{ex:33Ra}.

One can also observe that the cycle containing C in Figure \ref{fig: 33 sub} covers the corresponding cycle in Figure \ref{fig: 33hogmelcolor} twice, and two branches exiting this cycle get collapsed into one. Algebraically, this double cover corresponds to the cancellation $(1+q)/(1-q^2)=1/(1-q)$ in equation \eqref{eq: C} of Example \ref{ex: 33}.

\begin{figure}[ht!]

{\scalefont{0.9}
\begin{forest}
for tree={s sep= 11mm, l=0.5mm}
[{\scalefont{1.1}${\color{black}000}$}, name=root
[00$\times$,edge=blue, edge label={node[midway,left, blue,font=\scriptsize]{$1$}}
 [0$\times\times$,edge=blue, edge label={node[font=\scriptsize,midway,left, blue]{$t^{-1}$}}
 [$\times\times\times$,edge=blue, edge label={node[font=\scriptsize,midway,left, blue]{$t^{-2}$}}
	[$\times\times$,edge label={node[font=\scriptsize,midway,left]{$t^{2}+a$}}  
 	[$\times$,edge label={node[font=\scriptsize,midway,left]{$t+a$}} 
  	[$\emptyset$,edge label={node[font=\scriptsize,midway,left]{$1+a$}} ]]]]
 [$\times\times0$,edge=blue, edge label={node[font=\scriptsize,right,midway, blue]{$ qt^{-2}$}}
	[$\times0$,edge label={node[font=\scriptsize,midway,left]{$t+a$}} 
	[{\color{mygreen}{\fbox{0}=A}}, name=A,edge label={node[font=\scriptsize,midway,left]{$1+a$}} 
	[$\times$,edge=blue, edge label={node[midway,left, blue,font=\scriptsize]{$1$}} 
	[$\emptyset$,edge label={node[font=\scriptsize,midway,left]{$1+a$}} ]]] ]]
]
  [$0\times0$,edge=blue, edge label={node[font=\scriptsize,midway,right, blue]{$\; qt^{-1}$}}
 [$\times0\times$,edge=blue, edge label={node[font=\scriptsize,midway,left, blue]{$t^{-1}$}}
 [{\fbox{$0\times$}=B}, name=B, edge label={node[font=\scriptsize,midway,left]{$t+a$}}
	[$\times\times$,edge=blue, edge label={node[font=\scriptsize,midway,left, blue]{$t^{-1}$}}
 [$\times$,edge label={node[font=\scriptsize,midway,left]{$t+a$}} [$\emptyset$,edge label={node[font=\scriptsize,midway,left]{$1+a$}}]]]
[$\times0$,edge=blue, edge label={node[font=\scriptsize,midway,right, blue]{$qt^{-1}$}}
  [{\color{mygreen}{\fbox{0}=A}},name=A2,edge label={node[font=\scriptsize,midway,left]{$1+a$}}
[$\times$,edge=blue, edge label={node[midway,left, blue,font=\scriptsize]{$1$}}
	[$\emptyset$,edge label={node[font=\scriptsize,midway,left]{$1+a$}} ] ] ] ]
]]
[$\times00$,edge=blue, edge label={node[font=\scriptsize,midway,right, blue]{$\; qt^{-1}$}}
 [{\color{mygreen}{\fbox{00}=C}}, name=C,edge label={node[font=\scriptsize,midway,left]{$1+a$}}
[$0\times$,edge=blue, edge label={node[midway,left, blue,font=\scriptsize]{$1$}}
	 [$\times\times$,edge=blue, edge label={node[font=\scriptsize,midway,left, blue]{$t^{-1}$}}
	[$\times$,edge label={node[font=\scriptsize,midway,left]{$t+a$}}
	 [$\emptyset$,edge label={node[font=\scriptsize,midway,left]{$1+a$}}]]]
[$\times0$,edge=blue, edge label={node[font=\scriptsize,midway,right, blue]{$q t^{-1}$}} 
  	[{\color{mygreen}{\fbox{0}=A}},name=A3,edge label={node[font=\scriptsize,midway,left]{$1+a$}}
[$\times$,edge=blue, edge label={node[midway,left, blue,font=\scriptsize]{$1$}}
	 [$\emptyset$,edge label={node[font=\scriptsize,midway,left]{$1+a$}}]]]
]]] ] 
	 ]] ] 
\draw[->,blue] (root) .. controls +(south east:1cm) and +(east:1cm) .. node[near start,right]{$\, q$} (root);
\draw[->,blue] (A) .. controls +(south east:1cm) and +(east:1cm) ..  node[near start,right ]{$q$} (A);
\draw[->,blue] (A2) .. controls +(south east:1cm) and +(east:1cm) ..  node[near start,right ]{$q$} (A2);
\draw[->,blue] (A3) .. controls +(south east:1cm) and +(east:1cm) ..  node[near start,right]{$q$} (A3);
\draw[->,blue] (C) .. controls +(south east:1cm) and +(east:1cm) ..  node[near start,right ]{$q$} (C);
\end{forest}
}
\caption{
Decision tree for $(M,N)=(3,3)$ in Hogancamp-Mellit recursion.
Everywhere $\bx=\by$.}
\label{fig: 33hogmelcolor}
\end{figure}
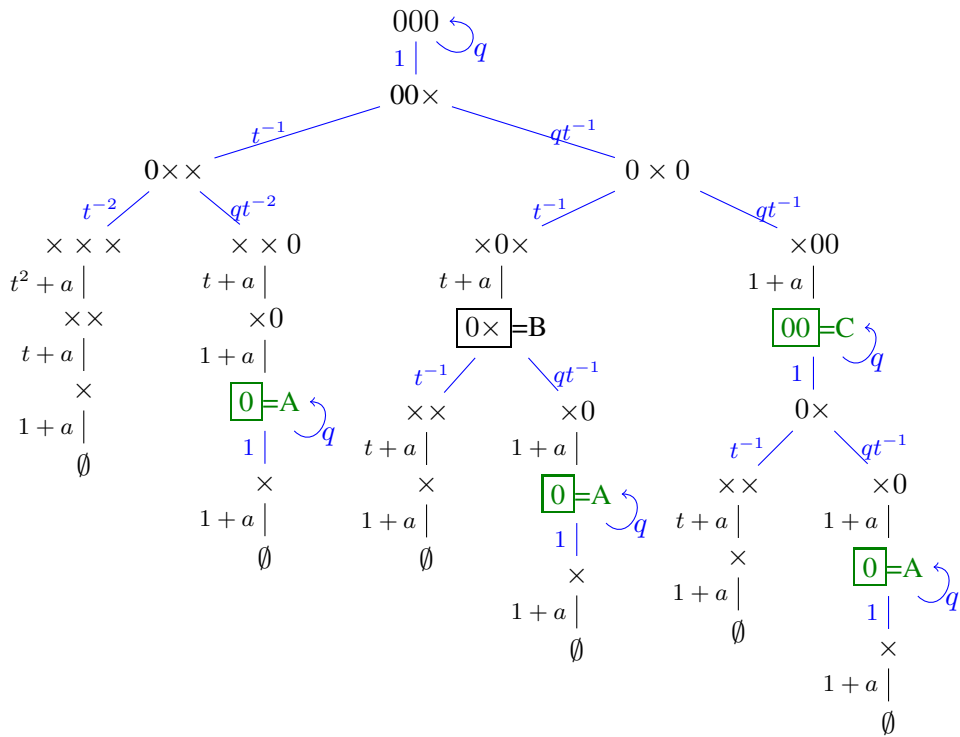


\end{example}

\section{Higher $a$-degrees}
\label{sec-new-higher-a}

The Mellit-Hogancamp construction involves an extra variable $a.$ One can enhance the generating series and the recurrence relations described in 
Section \ref{sec:recursion} to recover the full three variable functions in the following way.

\begin{definition}
Let $\Delta\in \IMN$ be an invariant subset. A number $k\in\overline{\Delta}$ is called a \textit{double cogenerator} of $\Delta$ if $k+N\in\Delta$ and $k+M\in\Delta.$ Let $\Cogen(\Delta)\subset\overline{\Delta}$ denote the set of all double cogenerators of $\Delta.$
\end{definition}

\begin{remark}
Note that we only consider \textbf{non-negative} double cogenerators. If $\Delta\in I_{0\ldots 0},$ then all cogenerators are positive, so it does not matter for such $\Delta$'s. However, this choice will matter for the recurrence relations.
\end{remark}

We will also need the following statistic:

\begin{definition}
Let $\Delta\in I_{N,M}$ be an invariant subset, and $k\in\BZ$ be an integer. We set
\begin{equation*}
\lambda_k(\Delta):=\sharp(\Ngen(\Delta)\cap [k+N+1,k+N+M])=\sharp(\Mgen(\Delta)\cap [k+M+1,k+N+M]).
\end{equation*}
\end{definition}

\begin{remark}
\label{rem: lambda zero}
Note that if $\Delta\in I_{\bu},$ where $\bu=(u_0,\ldots,u_{N+M-1}),$ then 
\begin{equation}\label{eq:lambda_0 vs lambda}
\lambda_0(\Delta)=\lambda(\rho(\Delta)) = \lambda(\bv),
\end{equation}
where $\bv=(u_1,\ldots,u_{N+M-1},1)$
as in Definition \ref{def:lambda}.
However, for any $\Delta$ we have $\lambda_{-1}(\Delta) = \lambda(\Delta)$.
\end{remark}

Now we are ready to define the enhancement of the counting function from Section \ref{sec:recursion}.

\begin{definition}
Let the power series $\hat{P}_{\bu}(q,t,a)$ be given by
\begin{equation*}
\hat{P}_{\bu}:=\sum\limits_{\Delta\in I_{\bu}} q^{\area(\Delta)}t^{\codinv(\Delta)}\prod\limits_{k\in\Cogen(\Delta)}\left(1+at^{\lambda_k(\Delta)}\right).
\end{equation*}
\end{definition}

\begin{remark}
We can think about the enumerating functions $\hat{P}_{\bu}$
in another way to
by expanding the parenthesis in the above product.
In this way one gets a summation over \textit{labeled} invariant subsets, i.e. invariant subsets with some (from none to all) of the double cogenerators labeled. This should be view as a generalization of the Schr\"oder  paths enumeration in the relatively prime case (see \cite{Hd08}).
\end{remark}


\begin{theorem}\label{Theorem: enhanced recursion binary N+M}
Let $\bu=(u_0,\ldots,u_{N+M-1})$ be an admissible sequence. Let also 
\begin{equation*}
\begin{aligned}
&\bv=(u_1,\ldots,u_{N+M-1},1),\\
&\bv'=(u_1,\ldots,u_{N+M-1},0).
\end{aligned}
\end{equation*}
The power series $\hat{P}_{\bu}$ satisfy the following recurrence relation:
\begin{equation*}
\hat{P}_{\bu}=
\begin{cases} 
q(P_{\bv}+P_{\bv'}),& \text{if} \ \ u_0=u_N=u_M=0,\\
qP_{\bv},& \text{if} \ \ u_0=0 \ \  \text{and} \ u_N+u_M=1,\\
q\left(1+at^{\lambda(\bv)}\right)P_{\bv},& \text{if} \ \ u_0=0 \ \  \text{and} \ u_N=u_M=1,\\
t^{\lambda(\bu)}P_{\bv},& \text{if} \ \ u_0=u_N=u_M=1.
\end{cases}
\end{equation*}
\end{theorem}

\begin{proof}
The proof is very similar to that of Theorem \ref{Theorem: recursion binary N+M}, but we need to account for double cogenerators.
As in the proof of  Theorem \ref{Theorem: recursion binary N+M}, consider the shift map $\rho$ of Definition \ref{def:rho}.
The sets $\Cogen(\Delta)$ and $\Cogen(\rho(\Delta))$ are in bijection unless $0$ is a double cogenerator of $\Delta$. This happens if and only if $u_0=0$ and $u_M=u_N=1$; and in this case
the extra term for $k=0$ is 
 $(1+at^{\lambda_{0}(\Delta)}) = (1+at^{\lambda(\bv)})$ according to equation \eqref{eq:lambda_0 vs lambda}.
\end{proof}

Similar to Section \ref{sec:comparison}, in order to match with the recursion of Hogancamp and Mellit, we switch to the adjusted statistics $\area'$, $\codinv'$, and $\dinv'.$

\begin{definition}
\label{def:hatQ}
As before, let ${\bu}=(u_0,\ldots,u_{N+M-1})\in\{0,1\}^{N+M}$ be an admissible binary sequence. The generating series $\hat Q_{\bu}(q,t,a)$ is defined by
\begin{equation*}
\hat{Q}_{\bu}(q,t,a):=\sum\limits_{\Delta\in I_{\bu}} t^{-\codinv'(\Delta)}q^{\area'(\Delta)}\prod\limits_{k\in\Cogen(\Delta)}\left(1+at^{-\lambda_k(\Delta)}\right).
\end{equation*}


We also set
\begin{equation*}
\hat{Q}_{{\bv},{\bw}}(q,t,a):=\hat{Q}_{\bu}(q,t,a),
\end{equation*}
where the sequences ${\bv}=(v_0,\ldots,v_{M-1})\in\{0,\bullet,\times\}^M$ and ${\bw}=(w_0,\ldots,w_{N-1})\in\{0,\bullet,\times\}^N$ are determined in the same way as in Definition \ref{Definition: v,w from u},
i.e. $(\bv,\bw) = \mathbf{b}(\bu)$.
\end{definition}

Note that similar to arguments in Section \ref{sec:comparison}, one gets
\begin{equation}\label{formula: hatQ vs hatP}
\hat{Q}_{\zeroMN}(q,t,a)=q^{-N-M}\hat{P}_{\zeroMN}(q,t^{-1},a).
\end{equation}

\begin{theorem}\label{Theorem: recursion hatQ_v,w}
The following recursion holds:
\begin{equation*}
\begin{aligned}
&\hat{Q}_{0{\bv},0{\bw}}=t^{-|{\bv}|}\hat{Q}_{{\bv}\times,{\bw}\times}+qt^{-|{\bv}|}\hat{Q}_{{\bv}0,{\bw}0},\\
&\hat{Q}_{\times{\bv},0{\bw}}=\hat{Q}_{{\bv}\times,{\bw}\bullet},\\
&\hat{Q}_{0{\bv},\times{\bw}}=\hat{Q}_{{\bv}\bullet,{\bw}\times},\\
&\hat{Q}_{\times{\bv},\times{\bw}}=(t^{|{\bv}|}+a)\hat{Q}_{{\bv}\bullet,{\bw}\bullet},\\
&\hat{Q}_{\bullet{\bv},\bullet{\bw}}=\hat{Q}_{{\bv}\bullet,{\bw}\bullet},\\
\end{aligned} 
\end{equation*}
\end{theorem}

\begin{proof}
The proof proceeds the same way as in Theorem \ref{Theorem: recursion Q_v,w}, with the exception for the fourth relation 
$$
\hat{Q}_{\times{\bv},\times{\bw}}=(t^{|{\bv}|}+a)\hat{Q}_{{\bv}\bullet,{\bw}\bullet}.
$$ 
The corresponding relation for $\hat{P}_{\bu}(q,t,a)$ is 
$$
\hat{P}_{\bu}(q,t,a)=q\left(1+at^{\lambda(\bz)}\right)P_{\bz},
$$
where $u_0=0$ and $u_N=u_M=1$, and $\bz=(u_1, \ldots, u_{N+M-1}, 1)$.
This differs from the corresponding relation for $P_{\bu}(q,t)$ by the factor of $\left(1+at^{\lambda(\bz)}\right),$ which after switching $t$ to $t^{-1}$ becomes $\left(1+at^{-\lambda(\bz)}\right).$ Therefore, one should multiply the corresponding relation 
$$
Q_{\times\bv,\times\bw}=t^{|\bv|}Q_{\bv\bullet,\bw\bullet}
$$ 
by $\left(1+at^{-|\bv|}\right),$ which yields the desired relation.
\end{proof}

\begin{corollary}\label{cor:R vs hatQ}
Let $\bx=\phi(\bv)$ and $\by=\phi(\bw).$ 
Then
\begin{equation*}
R_{\bx,\by}(q,t,a)
=\hat{Q}_{\bv,\bw}(q,t,a).
\end{equation*}
\end{corollary}

The proof is almost identical to  Theorem \ref{thm:ignore bullets} and we leave it to the reader.

\begin{example}
\label{ex:33Ra} Let us use the Hogancamp-Mellit recursion to compute $R_{000,000}$ (see Figure \ref{fig: 33hogmelcolor} for the decision tree):
\begin{align*}
&R_{000,000}(q,t,a) &=&
\frac{1}{1-q}\left( \left[(t^2+a) + q(t+a)\right]t^{-3}(t+a)(1+a)
\quad +
\right.&
\\
&&&
\left.
\quad
 (q+q^2) t^{-3} (t+a)(1+a) R_{0,0}(q,t,a) 
+ q^2 t^{-2}(1+a)R_{00,00}(q,t,a)  \right)
&
\\
&&=&
\frac{1}{1-q}\left[
(t^2+a) + q(t+a)
\right]
t^{-3}(t+a)(1+a)
\quad +&
\\
&&&
\qquad
\frac{1}{(1-q)^2}\left(
q+2q^2
\right)
 t^{-3} (t+a)(1+a)^2
\quad +&
\\
&&&
\qquad
\frac{1}{(1-q)^3}\left(
 q^3 t^{-3}
\right)
(1+a)^3
&
\end{align*}
\end{example}

\begin{example}
\label{ex:33hatP}
One can also modify Example \ref{ex: 33} to include higher powers of $a$ and compute $\hat{P}_{000000}(q,t,a).$ We first compute the value of the loop
$$
A=q^2t^2(1+a)+Aq,\quad A=\frac{q^2t^2(a+1)}{1-q}.
$$
Next we compute the values of $B$ and $C:$  
\begin{align*}
B&=q^3t(1+at)(1+a)+q^2(1+a)A=q^3t(1+at)(1+a)+\frac{q^4t^2(1+a)^2}{1-q},
\\
C&=q^4t^2(1+at)(1+a)+q^3t(1+a)A+q^2tB+q^2C\\
\end{align*}
hence
\begin{align*}
C&=\frac{q^4t^2(1+at)(1+a)+q^3t(1+a)A+q^5t^2(1+at)(1+a)+q^4t(1+a)A}{1-q^2}\\
&=\frac{q^4t^2(1+at)(1+a)}{1-q}+\frac{q^5t^3(1+a)^2}{(1-q)^2}.
\end{align*}
Finally,
\begin{align*}
(1-q)\hat{P}_{000000}(q,t,a)&=q^6(1+at^2)(1+at)(1+a)+q^5(1+at)(1+a)A \quad +\\
&\ \ \ \ \ \ q^4(1+at)B+q^4(1+a)C\\
&=q^6(1+at^2)(1+at)(1+a)+q^7t(1+at)^2(1+a) \quad +\\
&\ \ \ \ \ \ \frac{q^7t^2(1+at)(1+a)^2+2q^8t^2(1+at)(1+a)^2}{1-q} \quad +\\
&\ \ \ \ \ \ \frac{q^9t^3(1+a)^3}{(1-q)^2}.
\end{align*}

Note that according to Corollary \ref{cor:R vs hatQ} and Formula \ref{formula: hatQ vs hatP} one should get 

\begin{equation*}
R_{000,000}(q,t,a)=\hat{Q}_{000,000}(q,t,a)=q^{-6}\hat{P}_{000000}(q,t^{-1},a).
\end{equation*}

Indeed, one gets

\begin{align*}
\frac{(1-q)\hat{P}_{000000}(q,t^{-1},a)}{q^6}&=(1+at^{-2})(1+at^{-1})(1+a)+qt^{-1}(1+at^{-1})^2(1+a) \quad +\\
&\ \ \ \ \ \ \frac{qt^{-2}(1+at^{-1})(1+a)^2+2q^2t^{-2}(1+at^{-1})(1+a)^2}{1-q} \quad +\\
&\ \ \ \ \ \ \frac{q^3t^{-3}(1+a)^3}{(1-q)^2}\\
&=t^{-3}(t+a)(1+a)\left[(t^2+a)+q(t+a)\right] \quad +\\
&\ \ \ \ \ \ \frac{qt^{-3}(t+a)(1+a)^2(1+2q)}{1-q}+\frac{q^3t^{-3}(1+a)^3}{(1-q)^2},
\end{align*}
which matches the computation in Example \ref{ex:33Ra}.
Observe the $q, t$-symmetry of 
\begin{multline*}
t^3 q^{-6}(1-q)^3 \hat{P}_{000000}(q, t^{-1},a)
=   
(1+a)\left(
a^2+
a(t+q+t^2+q^2+q^2t^2+qt-2qt^2-2q^2t) +
\right. \\
\left.
(q^3t^2+q^2t^3 - 2q^3t-2qt^3 + q^3+t^3 +q^2t+qt^2-
 2q^2t^2 +qt) 
\right).
\end{multline*}
\end{example}

\end{document}